\RequirePackage{fix-cm}
\documentclass[smallextended,numbook,envcountsame]{svjour3} 
\smartqed  
\usepackage{graphicx}
\usepackage[misc]{ifsym}

\usepackage{amssymb,amsmath}

\usepackage{dsfont}
\usepackage{enumitem}
\DeclareMathAlphabet{\mathpzc}{OT1}{pzc}{m}{it}
\usepackage{scalerel,stackengine}
\stackMath
\newcommand\reallywidehat[1]{%
\savestack{\tmpbox}{\stretchto{%
  \scaleto{%
    \scalerel*[\widthof{\ensuremath{#1}}]{\kern.1pt\mathchar"0362\kern.1pt}%
    {\rule{0ex}{\textheight}}
  }{\textheight}%
}{2.4ex}}%
\stackon[-6.9pt]{#1}{\tmpbox}%
}
\newcommand{\eps}{\varepsilon}
\renewcommand{\P}[1]{\ensuremath{\mathbb P \left[ #1 \right]}}
\newcommand{\E}[1]{\ensuremath{\mathbb E \left[ #1 \right]}}

\newcommand{\Leb}{\operatorname{Leb}}
\newcommand{\Q}{\mathds Q}
\newcommand{\R}{\mathds R}

\newcommand{\N}{\mathds N}
\newcommand{\Nz}{\mathds N_0}

\newcommand{\mC}{\mathcal C}
\newcommand{\mCz}{\mC [0,\infty)^{\Nz}}
\newcommand{\mCo}{\mC [0,\infty)^\N}
\newcommand{\F}{\mathcal F}
\newcommand{\I}{\mathds 1}
\newcommand{\cdl}{c\`{a}dl\`{a}g }\usepackage{ifsym}
\newcommand{\DC}{D((0,1),\mC[0,\infty))}
\newcommand{\DI}{D^{\uparrow}}
\newcommand{\DIL}{D^{\uparrow}_2}
\newcommand{\CL}{\mathcal{C}L_2^{\uparrow}}
\newcommand{\Li}{L_2^{\uparrow}}
\newcommand{\ltp}{L_{2+}^{\uparrow}}

\newcommand{\St}{\mathrm {St}}
\newcommand{\p}{\mathbb{P}}
\newcommand{\e}{\mathbb{E}}
\newcommand{\coal}{{\bf Coal}}
\newcommand{\coalex}{{\bf Coal}^{\mathrm{ex}}}
\newcommand{\B}{\mathcal{B}}

\newcommand{\bE}{{\bf E}}
\newcommand{\bF}{{\bf F}}
\newcommand{\cP}{\mathcal P}
\newcommand{\bG}{{\bf G}}
\newcommand{\Y}{\mathpzc{Y}}
\newcommand{\W}{\mathpzc{W}}
\newcommand{\X}{\mathpzc{X}}
\newcommand{\mZ}{\mathcal{Z}}

\newcommand{\G}{\mathcal G}
\renewcommand{\H}{\mathcal H}
\newcommand{\Gl}{\mathrm{Gl}}
\newcommand{\T}{\mathrm{T}}
\newcommand{\ltz}{L_2^0}

\newcommand{\pr}{pr}
\newcommand{\law}{Law}
\newcommand{\spann}{span}

\begin{document}

\title{On Conditioning Brownian Particles to Coalesce
}

\author{Vitalii Konarovskyi         \and
        Victor Marx 
}

\institute{V. Konarovskyi \\
\Letter \; vitalii.konarovskyi@math.uni-bielefeld.de \at
Fakult\"{a}t f\"{u}r Mathematik, Bielefeld Universit\"{a}t, Universit\"{a}tsstra{\ss}e 25, 33615 Bielefeld, Germany.\\
              Fakult\"{a}t f\"{u}r Mathematik und Informatik, Universit\"{a}t Leipzig, Augustusplatz 10, 04109 Leipzig, Germany. \\
              Institute of Mathematics of NAS of Ukraine, Tereschenkivska st. 3, 01024 Kiev, Ukraine.      
           \and
           V. Marx \\ 
           \Letter \; marx@math.uni-leipzig.de
          \at
              Fakult\"{a}t f\"{u}r Mathematik und Informatik, Universit\"{a}t Leipzig, Augustusplatz 10, 04109 Leipzig, Germany.             
}

\date{\today}

\maketitle

\begin{abstract}
We introduce the notion of a conditional distribution to a zero-probability event in a given direction of approximation, and prove that the conditional distribution of a family of independent Brownian particles to the event that their paths coalesce after the meeting coincides with the law of a modified massive Arratia flow, defined in~\cite{Konarovskyi:AP:2017}.
\end{abstract}

\section{Introduction}

One of classical systems of interacting particles is the {\it Arratia flow} or {\it coalescing Brownian particles}, proposed by R.~Arratia in~\cite{Arratia:1979} (see also~\cite{Berestycki:2015,Le_Jan_Raimond:2004,Le_Jan:2004}). It is the family of one-dimensional Brownian motions with the same diffusion rate starting at every point of the real line and moving independently until their meeting. When two particles collide, they coalesce and move together. The model was obtained as a scaling limit of a continuous analog of a family of coalescing random walks on the real line, and the initial interest of the study was its connection with a voter model~\cite{Arratia:1979,Arratia:1981}. Later the Arratia flow and its generalization, Brownian web~\cite{Fontes:2004}, appear as scaling limits of seemingly disconnected models like true self-repelling motion~\cite{Toth:1998}, Hastings-Levitov planer aggregation models~\cite{Norris:2012}, oriented percolation~\cite{Sarkar:2013}, isotropic stochastic flows of homeomorphisms in $\R $~\cite{Piterbarg:1998}, solutions to evolutionary stochastic differential equations~\cite{Dorogovtsev:2004}, etc. In particular, this leads to the intensive study of the properties of the Arratia flow. We refer to~\cite{Tsirelson:2004,Fontes:2004,Riabov:2017,Dorogovtsev:2020,Shamov:2011,Vovchanskii:2012,Glinyanaya:2018,Dorogovtsev:2010,Tribe:2011,Munasinghe:2006,Korenovska:2017} for more details.

However the classical Arratia flow does not take into account the physical characteristics of particles like mass, spin, charge, etc., which can influence the particle behavior. In~\cite{Konarovskyi:AP:2017,Konarovskyi:CPAM:2019}, the first author proposed a physical improvement called a {\it modified massive Arratia flow} (shortly MMAF), where the diffusion rate of particles depends inversely proportional on their mass. More precisely, every particle carries a mass that obeys the conservation law, i.e., the mass of a new particle that appeared after the coalescing equals the sum of the colliding particles. This type of interaction makes the particle system more natural from a physical point of view and leads to a new local phenomena~\cite{Konarovskyi:EJP:2017}. It turns out that the MMAF is closely related with the geometry of the Wasserstein space of probabilities measures on the real line~\cite{Konarovskyi:CPAM:2019} and also is a non-trivial solution to the Dean-Kawasaki equation for supercooled liquids appearing in macroscopic fluctuation theory or models for glass dynamics in non-equilibrium statistical physics~\cite{Bertini:2015,Torre:2015,Dean:1996,Delfau:2016,Donev_Fai:2014,Donev_Vanden_Eijnden:2014,Embacher:2018,Giacomin:1999,Kawasaki:1994,Kipnis:1989,Marconi:1999,Rotskoff:2018,Sadhu:2016,Spohn:1991,Velenich:2008,Konarovskyi:ECP:2018,Konarovskyi:DK:2020}. For the regularised versions of the Dean-Kawasaki equation see also~\cite{Cornalba:2019,Cornalba:2020,Gess:2019}. This makes the model of a reasonable candidate for a Brownian motion on the Wasserstein space. 

The main goal of this paper is to show that the MMAF appears by the conditioning of independent Brownian particles (more precisely, a cylindrical Wiener process) to the event that particle paths ``coalesce'' after their meeting. 
To be more precice, we will justify that the conditional law of a cylindrical Wiener process in $L_2[0,1]$ starting at some non-decreasing function $g$ to the event of coalescence is the law of a MMAF. But we will pay the prize of having to investigate more carefully the notion of conditional law to a zero-probability event, allowing to define it only in some directions of approximation.  
First of all, this observation would explain some similarities of the particle model with a Wiener process in the Euclidian space. For instance, the rate function in the large deviation principle for the MMAF has a similar form as the rate function for a usual Wiener process (see~\cite[Theorem~2.1]{Konarovskyi:2014} for a finite particle system and~\cite[Theorem~1.4]{Konarovskyi:CPAM:2019} for the MMAF starting from all points of an interval). Secondary, we hope this result will shed some light on the uniqueness of the distribution of the MMAF, which is one of the biggest problems.

We first introduce a definition of a conditional distribution along a direction, which allows to interprete a value of the commomly-used  notion of regular conditional probability at a fixed point (see e.g~\cite[Theorem I.3.3]{Ikeda:1989} and~\cite[Theorem 6.3]{Kallenberg:2002} for the existence of the regular conditional probability).

Let $\bE$ be a Polish space, $\B(\bE)$ denote the Borel $\sigma$-algebra on $\bE$ and $\cP(\bE)$ be the space of probability measures on $(\bE,\B(\bE))$ endowed with the topology of weak convergence.  In general, given a random element $X$ in $\bE$ and $C \in \B(\bE)$ such that $\P{X \in C}=0$, defining the conditional probability $\P{X \in \cdot | X \in C}$ has no sense if we consider $\{X \in C  \}$ as an isolated event. However, one can make a proper definition with the help of regular conditional probability if $C$ is given by  $C= \T^{-1} (\{ z_0 \})$, where $z_0$ belongs to a metric space $\bF$ and $\T: \bE \to \bF$ is some measurable map.  Let $p: \B(\bE) \times \bF \to [0,1]$ be a regular conditional probability\footnote{See Definition~\ref{def:regular_conditional_probability} in appendix.} of $X$ given $\T(X)$. If $p(\cdot,z)$, $z\in\bF$, is continuous in $z_0$, then one can define $\P{X \in C}$ to be equal to $p(\cdot,z_0)$. But in general case, the regular conditional probability $p$ is well-defined for only $\p^{\T(X)}$-almost every $z \in \bF$, where $\p^{\T(X)}$ denotes the law of $\T(X)$. Therefore, we will introduce a notion of the value of $p$ at a fixed point along a (random) direction. 

\begin{definition}
\label{def:value_along_xi_n}
Let $\{ \xi^n\}_{n\geq 1}$ be a sequence of random elements in $\bF$ such that 
\begin{itemize}[noitemsep]
\item[(B1)] for each $n \geq 1$, the law of $\xi^n$ is absolutely continuous with respect to the law of $\T(X)$;
\item[(B2)] $\{\xi^n\}_{n\geq 1}$ converges  in distribution  to $z_0$ in $\bF$. 
\end{itemize} 
A probability measure $\nu$ on $(\bE,\B(\bE))$ is \emph{the value of the conditional distribution of $X$ to the event $\{ \T(X)=z_0\}$ along the sequence  $\{ \xi^n \}$} if for every $f \in \mC_b(\bE)$ 
\begin{equation} 
  \label{equ_value_of_p_along_xi_n}
  \E{\int_{\bE} f(x) p(\mathrm dx, \xi^n)} \to \int_{\bE} f(x) \nu (\mathrm dx), \quad n\to\infty,
\end{equation}
where $p$ is a regular conditional probability of $X$ given $\T(X)$. 
We denote this measure by $\nu= \law_{\{\xi^n\}}(X | \T(X)=z_0)$.
\end{definition}

We remark that the measure $\nu$ does not depend on the version of the regular conditional probability $p$. In Section~\ref{sec:conditional_distribution}, we explain that the above definition generalizes the case where $p$ is continuous at $z_0$ and that it is very close to the intuitive definition of the conditional probability $\P{ X \in \cdot \ | X \in C }$ by approximation of the set $C$. 
Furthermore, we introduce in Section~\ref{sec:conditional_distribution} a method to construct~$\nu$.

In order to formulate the main result of the paper, we remind the definition of the MMAF\footnote{For some properties of the MMAF see~\cite{Konarovskyi:TVP:2010,Konarovskyi:2014,Konarovskyi:AP:2017,Konarovskyi:EJP:2017,Marx:2018,Konarovskyi:CPAM:2019}.}. Let $\DC$ denote the space of \cdl functions from $(0,1)$ to $\mC([0,\infty), \R)$.
Let $g:[0,1] \to \R$ be a non-decreasing \cdl function such that $\int_0^1 |g(u)|^p \mathrm du < \infty$ for some $p>2$. 

\begin{definition}
\label{def:mmaf}
A random element $\Y=\{\Y(u,t),\ u \in (0,1),\ t \in [0,\infty)\}$ in the space $\DC$ is called \textit{modified massive Arratia flow} (shortly MMAF) starting at $g$ if it satisfies the following properties
 \begin{enumerate}[noitemsep]
    \item[(E1)] for all $u \in (0,1)$ the process $\Y(u,\cdot )$ is a continuous square-integrable martingale with respect to the filtration 
\begin{align}
\label{filtration_FtY}
	\F_t^\Y=\sigma(\Y(v,s),\ v \in (0,1),\ s\leq t), \quad t\geq 0;
\end{align}

    \item[(E2)] for all $u \in (0,1)$, $\Y(u,0)=g(u)$;

    \item [(E3)] for all $u<v$ from $(0,1)$ and $t\geq 0$, $\Y(u,t)\leq \Y(v,t)$;

    \item [(E4)] for all $u,v \in (0,1)$, the joint quadratic variation of $\Y(u,\cdot)$ and $\Y(v, \cdot )$ is 
      \[
	\left\langle \Y(u,\cdot ),\Y(v,\cdot ) \right\rangle_t =\int_{ 0 }^{ t } \frac{ \I_{\left\{ \tau_{u,v}\leq s \right\}} }{ m(u,s) } \mathrm ds, \quad t\geq 0, 
      \]
      where $m(u,t)=\Leb\left\{ v:\ \exists s\leq t,\ \Y(v,s)=\Y(u,s) \right\}$ and $\tau_{u,v}=\inf\{ t:\ \Y(u,t)=\Y(v,t) \}$.
  \end{enumerate}
\end{definition}
Intuitively, the massive particles $\Y(u,\cdot)$, for each $u \in (0,1)$, evolve like independent Brownian particles with diffusion rates inversely proportional to their masses, until two of them collide. 
 When two particles meet, they coalesce and form a new particle with the mass equal to the sum of masses of the colliding particles.

The random element $\Y$ can be identified with an $\Li$-valued process $\Y_t$, $t \geq 0$, where $\Li$ is the subset of $L_2[0,1]$ consisting of all functions which have non-decreasing versions. There exists a cylindrical Wiener process $\W$ in $L_2[0,1]$ starting at $g$ such that 
\begin{equation} 
  \label{intro:equ_couple}
  \Y_t=g+\int_{ 0 }^{ t } \pr_{\Y_s} \mathrm d\W_s, \quad t\geq 0,
\end{equation}
where for any $f \in \Li$, $\pr_f$ is the orthogonal projection operator in $L_2[0,1]$ onto the subspace of $\sigma(f)$-measurable functions. Those results will be recalled with further details and references in Section~\ref{sec:properties_statement}.

Our main results consists in the construction of the following objects and in the following theorem.

\begin{enumerate}[noitemsep]

\item[(S1)] We start from $\Y$, a MMAF starting at a strictly increasing map $g$. 

\item[(S2)] Thus there exists a cylindrical Wiener process $\W$ in $L_2[0,1]$ starting at $g$ satisfying~\eqref{intro:equ_couple}. $\Y$ can be seen as the coalescing part of $\W$. 

\item[(S3)] Given $\X=(\Y,\W)$, we  decompose $\W$ into $\Y$ and a non-coalescing part $\T(\X)$, so that $\W$ is completely determined by $\Y$ and $\T(\X)$. We postpone to Section~\ref{parag:construction_map_T} the precise definition of the map $\T$. 
We are interested in the conditional distribution of $\X$ to the event $\{ \T(\X)=0\}$, which is the event where $\W$ coincides with its coalescing part $\Y$.

\item[(S4)] For every $n \geq 1$, $\xi^n$ is defined as a sequence $\{\xi^n_j\}_{j \geq 1}$ of independent Ornstein-Uhlenbeck processes such that $\{ \xi^n\}_{n\geq 1}$ converges in distribution to zero in the space $\mC[0,\infty)^{\N}$, equipped with the product topology, and the law of $\xi^n$ is absolutely continuous with respect to the law of $\T(\X)$, which is the law of a sequence of independent standard Brownian motions. 
\end{enumerate}

\begin{theorem}
\label{theo:main}
The value of the conditional distribution of $\X=(\Y,\W)$ to the event $\{ \T(\X)=0\}$ along $\{\xi^n\}$ is the law of $(\Y,\Y)$.
\end{theorem}

Unfortunatly, we can not prove the result for any sequence $\{\xi^n\}$ satisfying (B1)-(B2), and this seems  to be not achievable and possibly even not true. Nevertheless, a sequence of Ornstein-Uhlenbeck processes  seems a  reasonable choice of $\{\xi^n\}$ satisfying (B1)-(B2).
We refer to Theorem~\ref{theo:infinite_dim} for a more precise statement after having carefully defined $\T$ and $\{ \xi^n\}_{n\geq 1}$ among others.

Our second result is the fact that $(\Y,\W)$ coupled by equation~\eqref{intro:equ_couple} is uniquely determined by the law of $\Y$. It does not impy the uniqueness of the distribution of $\Y$. However, we hope that it could be a first step in the proof that $\Y$ is a unique solution the the SDE~\eqref{intro:equ_couple}.
\begin{theorem}
\label{theo:coupling}
Let $\Y_t$, $t \geq 0$, be a MMAF starting at $g$. Let $\W$ and $\widetilde{\W}$ be cylindrical Wiener processes in $L_2$ starting at $g$ and such that $(\Y, \W)$ and $(\Y, \widetilde{\W})$ satisfy equation~\eqref{intro:equ_couple}. Then $\law (\Y, \W)= \law (\Y, \widetilde{\W})$. 
\end{theorem}

Theorem~\ref{theo:coupling} has an interest which is independent of the conditional distribution problem, but it is proved using the same techniques as for Theorem~\ref{theo:main}. Moreover, as a corollary, one can see that  steps (S1) and (S2) in the statement of the main result can be replaced  by starting from any pair $(\Y,\W)$ coupled by~\eqref{intro:equ_couple}, which is a  stronger result.

\textbf{Content of the paper.}
In Section~\ref{sec:conditional_distribution}, we  propose a method for the construction of a conditional distribution according to Definition~\ref{def:value_along_xi_n}. In Section~\ref{sec:properties_statement}, we recall needed properties of the MMAF and define the non-coalescing map~$\T$, using a construction of an orthonormal basis in $L_2[0,1]$ which is tailored for the MMAF. Finally in that section, we state the main result in Theorem~\ref{theo:infinite_dim}.  Sections~\ref{sec:proof_main_theorem}, and~\ref{sec:coupling}   are devoted to the proofs of Theorem~\ref{theo:infinite_dim} and~Theorem~\ref{theo:coupling}, respectively.

\section{On conditional distributions}
\label{sec:conditional_distribution}

Definition~\ref{def:value_along_xi_n} is consistent with the continuous case. Indeed, if the map $z \mapsto p(\cdot,z)$ is continuous at $z_0$, then by the continuous mapping theorem $p(\cdot,z_0) = \law_{\{\xi^n\}}(X | \T(X)=z_0)$ for any sequence $\{ \xi^n\}_{n\geq 1}$ satisfying $(B1)$ and $(B2)$. Actually, it is an equivalence, as the following lemma shows.

\begin{lemma}\label{lem_continuity_and_convergence}
  Let $z_0$ belong to the support of $\p^{\T(X)}$. There exists a probability measure $\nu$
such that   $\nu= \law_{\{\xi^n\}}(X | \T(X)=z_0)$  along any sequence $\{\xi^n\}_{n\geq 1}$ satisfying (B1) and (B2) if and only if there exists a version of $p$ which is continuous at $z_0\in \bF$. In this case, $\nu$ is equal to the value of the continuous version of $p$ at $z_0$.
\end{lemma}

We postpone the proof of the lemma to Section~\ref{parag:proof_of_lem_continuity_and_convergence} in the appendix.

\begin{remark}
\label{rem:conditioning on a subset}
Definition~\ref{def:value_along_xi_n} extends the intuitive definition of the conditional distribution of $X$ given $\{X \in C\}$ as the weak limit 
\[
\P{ X \in \cdot \ | X \in C }= \lim_{ \eps \to 0 }\P{ X \in \cdot \ | X \in C_{\eps} },
\]
where $C$ is a closed subset of $\bE$ and $C_\eps$ denotes its $\eps$-extension, that is, 
$C_{\eps}=\left\{x \in \bE:\ d_{\bE}(C,x)<\eps \right\}$. We assume $\P{X \in C_\eps}>0$ for any $\eps>0$. Then $\T$ can be defined by $\T(x):=d_{\bE}(C,x)$. We note that $ \{X \in C\} = \{ \T(X)=0\}$ and $\{X \in C_\eps\} = \{ \T(X)<\eps\}$ for all $\eps >0$. 
The sequence $\{\xi^n\}$ could then be defined by
  \[
    \P{ \xi^n \in A }= \frac{1}{ \P{ \T(X)< \frac{1}{n} } }\int_{ A }   \I_{\left\{ x<\frac{1}{n} \right\}}\p^{\T(X)}(\mathrm dx), \quad A \in \B(\bE).
  \]
One can easily check that $\{\xi^n\}$ satisfies conditions (B1) and (B2) with $z_0=0$, and that
\begin{align*}
\E{\int_{\bE} f(x) p(\mathrm dx, \xi^n)}
= \int_{ \bE }   f(x)\P{ X \in \mathrm  dx |X \in C_{1/n} }.
\end{align*}
Therefore, the weak limit of the sequence $(
\P{ X \in \cdot \ |X \in C_{1/n} })_{n\geq 1}$ coincides with the measure $\law_{\{\xi^n\}}(X | \T(X)=0)$ if it exists.
\end{remark}

We next introduce an idea to build  a  conditional distribution of $X$ given $\{\T(X)=z_0\}$ along a sequence $\{\xi^n\}$.  The idea is to split the random element $X$ into two independent parts, $Y$ and $Z$, so that $Z$ has the same law as $\T(X)$. More precisely, we assume that there  exists a quadruple $(\bG,\Psi,Y,Z)$ satisfying the following conditions
\begin{enumerate}
\item [(P1)] $\bG$ is a measurable space;
\item [(P2)] $Y$ and $Z$ are independent random elements in $\bG$ and $\bF$, respectively;
\item [(P3)] $\Psi: \bG \times \bF \to \bE$ is a measurable map such that $\T(\Psi(Y,Z))=Z$ a.s.;
\item [(P4)] $X$ and $\Psi(Y,Z)$ have the same distribution.
\end{enumerate}

\begin{proposition}
\label{prop:existence_regular}
Let $(\bG,\Psi,Y,Z)$ be a quadruple satisfying (P1)-(P4). 
The map $p$ defined by
\begin{equation}
\label{eq_def_p}
p(A,z):=\P{\Psi(Y,z) \in A }, \quad A \in \B(\bE), \ z \in \bF
\end{equation}
is a regular conditional probability of $X$ given $\T(X)$. 

Moreover, if $\{\xi^n\}_{n \geq 1}$ is a sequence of random elements in $\bF$ independent of $Y$ and satisfying (B1) and (B2) of Definition~\ref{def:value_along_xi_n}, 
then  $\Psi(Y, \xi^n)$ converges in distribution to the measure $\law_{\{\xi^n\}}(X | \T(X)=z_0)$.
\end{proposition}

\begin{proof} 
Since $\Psi$ is measurable, $p$ defined by~\eqref{eq_def_p} satisfies properties (R1) and (R2) of Definition~\ref{def:regular_conditional_probability}. Furthermore, for every  $A \in \B(\bE)$ and $B \in \B(\bF)$
\begin{align*}
 \P{X \in A,\ \T(X) \in B} 
&\overset{(P4)}{=} \P{\Psi(Y,Z) \in A,\ \T(\Psi(Y,Z)) \in B} \\
&\overset{(P3)}{=} \P{\Psi(Y,Z) \in A,\ Z \in B}  \\
&\overset{(P2)}{=} \int_B p(A,z)  \p^Z (\mathrm dz).
\end{align*}
Moreover, since $X$ and $\Psi(Y,Z)$ have the same law, $\T(X)$ and $Z=\T(\Psi(Y,Z))$ have the same law too, so $\p^Z=\p^{\T(X)}$. This concludes the proof of (R3). 

Let $f \in \mC_b(\bE)$.
By~\eqref{eq_def_p} and Proposition~\ref{prop:kallenb}, we know that for any regular conditional probability $p$ of $X$ given $\T(X)$, the equality $\int_\bE f(x) p(\mathrm dx,z)= \E{ f(\Psi(Y, z))}$ holds for $\p^{\T(X)}$-almost all $z \in \bF$. It also holds $\p^{\xi^n}$-almost everywhere by Property (B1). 
By independence of $\xi^n$ and $Y$ and Fubini's theorem, 
\begin{align*}
\E{f(\Psi(Y, \xi^n))}
= \int_\bF \E{ f(\Psi(Y, z))}  \p^{\xi^n} (\mathrm dz)
&=  \int_\bF \int_\bE f(x) p(\mathrm dx,z)  \p^{\xi^n} (\mathrm dz).
\end{align*}
By~\eqref{equ_value_of_p_along_xi_n}, the last term tends to $\int_\bE f(x) \nu (\mathrm dx)$, where $\nu=\law_{\{\xi^n\}}(X | \T(X)=z_0)$. This concludes the proof of the convergence in distribution. 
\qed
\end{proof}

\section{Precise statement of the main result}
\label{sec:properties_statement}

In the introduction, we announced the construction of several objects, including a modified massive Arratia flow (MMAF) and a non-coalescing remainder map~$\T$. The main part of this construction will be the definition of an orthonormal basis of $L_2[0,1]$ which is tailored for the MMAF. 
In this section, we will follow the  steps (S1)-(S4) from the introduction and finally, we will state again Theorem~\ref{theo:main} in a more precise form, see Theorem~\ref{theo:infinite_dim}.

\subsection{MMAF and set of coalescing paths}
\label{parag:mmaf_coal}

In this section, we define the set $\coal$ of coalescing trajectories in an infinite-dimensional space and we recall important properties of the MMAF introduced in Definition~\ref{def:mmaf} to show that it takes values almost surely in $\coal$. Since they are not the central issue of this paper, the proofs of this section will be succinct, but we will refer to previous works or to the appendix for the detailled versions.

Fix $g$ belonging to the set $\ltp$ that consists of all non-decreasing \cdl functions $g:(0,1) \to \R$ satisfying $\int_0^1 |g(u)|^{2+\eps} \mathrm du <\infty$ for some $\eps >0$.  
Let $\St$ denote the set of non-decreasing step functions $f:[0,1) \to \R$ of the form
\begin{equation}
\label{step_function}
f = \sum_{j=1}^n f_j \mathds 1_{\pi_j},
\end{equation}
where $n \geq 1$, $f_1 < \dots < f_n$ and $\{\pi_1, \dots \pi_n\}$ is an ordered partition of $[0,1)$ into 
half-open intervals of the form $\pi_j=[a_j,b_j)$.
The natural number $n$ is denoted by $N(f)$ and is by definition finite for every $f \in \St$.
Recall that $L_2:=L_2[0,1]$ and that $\Li$ is the subset of $L_2$ 
 consisting of all functions which have non-decreasing versions.

\begin{definition}
\label{definition_coal}
We  define $\coal$ as the set of functions $y$ from $\mC([0,\infty),\Li)$ such that
\begin{enumerate}
\item[(G1)] $y$ has a version in $D((0,1), \mC[0,\infty))$, the space of \cdl functions from $(0,1)$ to $\mC([0,\infty), \R)$; 
\item [(G2)] $y_0= g$;
\item [(G3)] for each $t>0$, $y_t \in \St$;
\item [(G4)] for each $u,v \in (0,1)$ and $s\geq 0$, $y_s(u)=y_s(v)$ implies $y_t(u)=y_t(v)$ for every $t\geq s$;
\item [(G5)]  $t \mapsto N(y_t)$, $t\geq 0$, is a \cdl non-increasing integer-valued function with jumps of height one and which is constant equal to $1$ for sufficiently large time.  
\end{enumerate}
\end{definition}

We can interpret $y$ as a deterministic particle system, where $y_t(u)$, $t \geq 0$,  describes the trajectory of a particle labeled by $u$. Condition (G3) means that there is only a finite number of particles at each positive time. By Condition (G4), two particles coalesce when they meet. Moreover, by Condition (G5), there can be at most one coalescence at each time, and the number of particles is equal to one for large time.

Note that, according to Lemma~\ref{lemma:measurability_step} in appendix, the set $\coal$ is measurable in $\mC([0,\infty),\Li)$. We will also consider $\coal$ as a metric subspace of $\mC([0,\infty),\Li)$.

Recall the following existence property of modified massive Arratia flow. 

\begin{proposition}
\label{prop:exis_mmaf}
Let $g \in \ltp$. 
There exists a MMAF starting at $g$. 
\end{proposition}

\begin{proof}
See~\cite[Theorem 1.1]{Konarovskyi:EJP:2017}.
\qed
\end{proof}

Equivalently, we may also define a MMAF as an $\Li$-valued process, in the following sense. For every $f \in \Li$, $\pr_f$ denotes the orthogonal projection operator in $L_2$ onto the subspace of $\sigma(f)$-measurable functions. 

\begin{lemma}
\label{lemma:equiv_def_mmaf}
Let $g \in \ltp$ and  $\{\Y(u,t),\ u \in (0,1),\ t \in [0,\infty)\}$ be a MMAF starting at $g$. Then the process $\Y_t$, $t\geq 0$, defined by $\Y_t:=\Y(\cdot,t)$, $t\geq 0$, satisfies
\begin{enumerate}[noitemsep]
  \item [(M1)] $\Y_t$, $t\geq 0$, is a continuous $\Li$-valued process with $\E{ \|\Y_t\|_{L_2}^2 }<\infty$, $t\geq 0$;

  \item [(M2)] for every $h \in L_2$ the $L_2$-inner product $(\Y_t,h)_{L_2}$, $t\geq 0$, is a continuous square integrable martingale with respect to the filtration generated by $\Y_t$, $t\geq 0$, that trivially coincides with $(\F_t^{\Y})_{t\geq 0}$;

  \item [(M3)] the joint quadratic variation of $(\Y_t,h_1)_{L_2}$, $t\geq 0$, and $(\Y_t,h_2)_{L_2}$, $t\geq 0$, equals
  $\left\langle (\Y_{\cdot },h_1)_{L_2},(\Y_{\cdot },h_2)_{L_2} \right\rangle_t=\int_{ 0 }^{ t } (\pr_{\Y_s}h_1,h_2)_{L_2}\mathrm ds$, $t\geq 0$. 
\end{enumerate}
Furthermore, if a process $\Y_t$, $t\geq 0$, starting at $g$ satisfies (M1)-(M3), then there exists a MMAF $\{\Y(u,t),\ u \in (0,1),\ t \in [0,\infty)\}$ such that $\Y_t=\Y(\cdot,t)$ in $L_2$ a.s. for all $t \geq 0$. 
\end{lemma}

\begin{proof}
The first part of the statement follows directly  from Lemma~\ref{lem_finiteness_of_l2_norm_of_mmaf} in appendix, for Property (M1), and from~\cite[Lemma~3.1]{Konarovskyi:CPAM:2019}, for properties (M1) and (M2). 
As regards the second part of the lemma, it is proved in~\cite[Theorem~6.4]{Konarovskyi:EJP:2017}.
\qed
\end{proof}

According to Lemma~\ref{lemma:equiv_def_mmaf}, we may identify the modified massive  Arratia flow $\{\Y(u,t),\ u \in (0,1),\ t \in [0,\infty)\}$ and the $\Li$-valued martingale $\Y_t$, $t\geq 0$, using both notations for the same object. 

\begin{lemma}
\label{lem:mmaf_belongs_to_coal}
The process $\Y_t$, $t\geq 0$, belongs almost surely to $\coal$.
\end{lemma}

\begin{proof}
By construction, the process satisfies properties (G1) and (G2). Properties (G3) and (G4) were proved in~\cite{Konarovskyi:EJP:2017}, propositions~6.2 and~2.3 ibid, respectively. Property (G5) is stated in Lemma~\ref{lem_property_of_ny} in appendix. 
\qed
\end{proof}

\subsection{MMAF and cylindrical Wiener process}
\label{parag:mmaf_cylindrical}

The goal of this section is the precise construction of a cylindrical Wiener process $\W$ for which the equality~\eqref{intro:equ_couple} holds for a given MMAF $\Y$. This will complete step (S2) from the introduction.

For every $f \in \Li$, let $L_2(f)$ denote the subspace of $L_2$ consisting of $\sigma(f)$-measurable functions. In particular, if $f$ is of the form~\eqref{step_function}, then $L_2(f)$ consists of all step functions which are constant on each $\pi_j$. For any $f \in \Li$, let
$\pr_f$ (resp.  $\pr_{f}^{\bot}$)  denote the orthogonal projection in $L_2$ onto $L_2(f)$ (resp. onto $L_2(f)^{\bot}$). Moreover, 
for any progressively measurable process $\kappa_t$, $t\geq 0$, in $L_2$ and for any cylindrical Wiener process $B$ in $L_2$, we denote
\[
\int_0^t \kappa_s \cdot\mathrm dB_s:=\int_0^t K_s \mathrm dB_s.
\]
where $K_t=(\kappa_t,\cdot)_{L_2}$, $t\geq 0$.

\begin{proposition}
\label{prop_mmaf_and_cylindrical_brownian_motion}
Let $g \in \ltp$ and $\Y_t$, $t\geq 0$, be a MMAF starting at $g$. 
Let $B_t$, $t\geq 0$, be a cylindrical Wiener process  in $L_2$ starting at $0$ defined on the same probability space and independent of $\Y$. Then the process $\W_t$, $t \geq 0$, defined by 
  \begin{equation} 
  \label{equ_cylindrical_wiener_process_theta}
  \W_t :=\Y_t +\int_{ 0 }^{ t } \pr_{\Y_s}^{\bot} \mathrm dB_s, \quad t \geq 0,
  \end{equation}
  is a cylindrical Wiener process in $L_2$ starting at $g$, where equality~\eqref{equ_cylindrical_wiener_process_theta} should be understood\footnote{The process $\pr_{\Y_t}^\bot$, $t\geq 0$, does not take values in the space of Hilbert-Schmidt operators in $L_2$. Therefore, the integral $\int_{ 0 }^{ t } \pr_{\Y_s}^{\bot} \mathrm dB_s$ is not well-defined but $h\mapsto \int_{ 0 }^{ t } \pr_{\Y_s}^{\bot} h \cdot \mathrm dB_s$ is.}
   as follows:
  \begin{align*}
 \W_t (h):=(\Y_t,h)_{L_2} +\int_{ 0 }^{ t } \pr_{\Y_s}^{\bot} h \cdot \mathrm dB_s, \quad t \geq 0, \ h \in L_2. 
\end{align*}
Moreover, $(\Y,\W)$ satisfies equation~\eqref{intro:equ_couple}.
\end{proposition}

\begin{proof} 
It follows from Property (M3) and from~\cite[Corollary~2.2]{Gawarecki:2011} that   
 there exists a cylindrical Wiener process $\tilde{B}$ in $L_2$ starting at $0$ (possibly on an extended probability space also denoted  by $(\Omega,\F,\p)$) such that 
  \[
    \Y_t=g+\int_{ 0 }^{ t } \pr_{\Y_s} \mathrm d \tilde{B}_s, \quad t\geq 0. 
  \]
Moreover, we may assume that $\tilde{B}$ is independent of $B$. It is trivial that the map $\W_t:L_2 \to L_2(\Omega,\F,\p)$ defined by~\eqref{equ_cylindrical_wiener_process_theta} is linear.  Let $(\F_t)_{t\geq 0}$ be the natural filtration generated by $\tilde{B}$ and $B$. 
  Let us check that $\W_t(h)$, $t\geq 0$, is an $(\F_t)$-Brownian motion starting at $(g,h)_{L_2}$ with diffusion rate $\|h\|_{L_2}^2$ for any $h \in L_2$. Using the independence of $\tilde{B}$ and $B$, we have that $ \W_t (h)$, $t\geq 0$, is a continuous $(\F_t)$-martingale with quadratic variation 
  \begin{align*}
    \left\langle \W(h ) \right\rangle_t&= \int_{ 0 }^{ t } \|\pr_{\Y_s}h\|_{L_2}^2\mathrm ds+ \int_{ 0 }^{ t } \|\pr_{\Y_s}^{\bot}h\|_{L_2}^2\mathrm ds=\int_{ 0 }^{ t } \|h\|_{L_2}^2\mathrm ds=t\|h\|_{L_2}^2.  
  \end{align*}
This implies that $\W$ is a cylindrical Wiener process. 
  
Moreover, for every $h \in L_2$ and $t \geq 0$, 
\begin{align*}
\int_0^t \pr_{\Y_s} h \cdot \mathrm d \W_s
&= \int_0^t (\Y_s, \pr_{\Y_s} h)_{L_2} \mathrm ds
+ \int_0^t \pr_{\Y_s}^\bot \circ \pr_{\Y_s} h \cdot \mathrm d B_s \\
&= \int_0^t (  \Y_s, h)_{L_2} \mathrm ds = (  \Y_t, h)_{L_2} - (  g, h)_{L_2}.
\end{align*}
Therefore $\Y_t=g + \int_0^t \pr_{\Y_s} \mathrm d \W_s$, which is equality~\eqref{intro:equ_couple}.
\qed
\end{proof}

Note that it is not obvious whether each cylindrical Wiener process $\W$ in $L_2$ starting at $g$ and satisfying~\eqref{intro:equ_couple} is necessary of the form~\eqref{equ_cylindrical_wiener_process_theta}. Actually, this is the result of Theorem~\ref{theo:coupling} and will be proved in Section~\ref{sec:coupling}.

\subsection{Construction of non-coalescing remainder map}
\label{parag:construction_map_T}

\textit{Up to now and until the end of Section~\ref{sec:proof_main_theorem}, we fix  a strictly increasing function~$g$ in $\ltp$ and $\X:=(\Y, \W)$, where $\Y_t$, $t \geq 0$, is  a modified massive Arratia flow  starting at $g$ and $\W_t$, $t \geq 0$, is defined by~\eqref{equ_cylindrical_wiener_process_theta}.} In particular, the assumption on~$g$ implies that  $L_2(g)=L_2$. 
In this section, we consider step (S3) from the introduction.

Let us introduce for every $y \in \coal$ the corresponding coalescence times:
\begin{equation}
\label{def:tau}
\tau^{y}_k:= \inf \{ t\geq 0 :\ N(y_t) \leq k  \}, \quad k \geq 0.
\end{equation}
Since $g$ is a strictly increasing function, one has that $N(g)=+\infty$, and therefore, the family $\{\tau^y_k,\ k \geq 0\}$ is strictly decreasing for all $y \in \coal$, i.e.
\[
0 < \dots < \tau_2^y < \tau_1^y < \tau_0^y=+\infty,
\]
by Condition (G5). 

Now we are going to define an orthonormal basis $\{ e_{k}^y,\ k\geq 0 \}$ in $L_2$ which depends on $y \in \coal$.  Since $y_t$, $t\geq 0$, is an $L_2$-valued continuous function and $L_2(g)=L_2$ due to the strong increase of $g$, it is easily seen that the closure of $\bigcup_{ k=1 }^{ \infty } L_2(y_{\tau_{k}^y})$ coincides with $L_2$. Let $H_k^y$ be the orthogonal complement of $L_2(y_{\tau_{k}^y})$ in $L_2$, $k\geq 1$. 

\begin{lemma}
\label{lem:onb}
For each $y \in \coal$ there exists a unique orthonormal basis $\{e_l^y,\ l \geq 0\}$ of $L_2$ such that
\begin{enumerate}
  \item[1)] the family $\{ e_l^y,\ 0\leq l< k \}$ is a basis of $L_2(y_{\tau_{k}^y})$ for each $k\geq 1$;

\item[2)] $(e_l^y, \I_{[0,u]})_{L_2} \geq 0$ for every $u \in (0,1)$.
\end{enumerate}
Moreover, the family $\left\{ e_{l}^y,\ l\geq k \right\}$ is a basis of $H_k^y$ for each $k\geq 1$.
\end{lemma}

In other words, the map $t \mapsto \pr_{y_t}$ is a projection map onto a subspace which decreases from exactly one dimension whenever a coalescence of $y$ occurs, and the basis $\{e_l^y,\ l \geq 0\}$ is adapted to that decreasing sequence of subspaces.

\begin{proof}
Let us construct the family $\{ e_k^y,\ k\geq 0\}$ explicitly. Since $y_{\tau_1^y}$ is constant on $[0,1]$, the only choice is $e_0^y=\I_{[0,1]}$. 

We say that an interval $I$ is a step of a map $f$ if $f$ is constant on $I$ but not constant on any interval strictly larger than $I$. At time $\tau_{k}^y$ a coalescence occurs. 
So there exist $a<b<c$ such that  $[a,b)$ and $[b,c)$ are steps of $y_{\tau_{k+1}^y}$, and $[a,c)$ is a step of $y_{\tau_{k}^y}$. We call $b$ \textit{the coalescence point} of $y_{\tau_{k}^y}$.
The only possible choice for $e_k^y$ so that it has norm $1$, it belongs to $L_2(y_{\tau_{k+1}^y})$, it is orthogonal to every element of $L_2(y_{\tau_{k}^y})$ and it satisfies Condition 2) is:
\begin{align}
\label{eky}
e_k^y= \frac{1}{\sqrt{c-a}}  \left(\sqrt{\frac{c-b}{b-a}} \I_{[a,b)}
-\sqrt{\frac{b-a}{c-b}} \I_{[b,c)}
  \right).
\end{align}
Since $\overline{\bigcup_{ k=1 }^{ \infty } L_2(y_{\tau_k^y})}=L_2$, we get that $\{ e_k^y,\ k\geq 0\}$ form a basis of $L_2$.

The last part of the statement follows from the fact that  for each $k \geq 1$, $H_k^y = L_2(y_{\tau_{k}^y})^{\bot}$.
\qed
\end{proof}

\begin{remark} 
  \label{rem_measurability_of_ey_inf_case}
  The construction of the basis $\left\{ e_k^y,\ k\geq 0 \right\}$ in the above proof  easily implies that the map $\coal\ni y\mapsto e_k^y \in L_2$ is measurable for any $k\geq 0$, where $\coal$ is endowed with the  induced topology of $\mC([0,\infty),\Li)$. Moreover, by~\eqref{eky},
  for every $k\geq 1$, $e_k^y$ is uniquely determined by $y_{\cdot \wedge \tau_k^y}$.
\end{remark}

According to step (S3), given $\X=(\Y,\W)$, we will define now the non-coalescing part $\T(\X)$ of $\W$.
 Note that  $\tau_{k}^{\Y}$ are $(\F_t^{\Y})$-stopping times for all $k\geq 0$, where $(\F^{\Y}_t)_{t\geq 0}$ is the complete right-continuous filtration generated by the MMAF $\Y$. Furthermore, Remark~\ref{rem_measurability_of_ey_inf_case} yields that $e_{k}^{\Y}$ is an $\F^{\Y}_{\tau_k^{\Y}}$-measurable random element in $L_2$. \textit{To simplify the notation, we will write $e_k$ and $\tau_k$ instead of $e_{k}^{\Y}$ and $\tau_{k}^{\Y}$, respectively.}

Recall that $\W$ is defined by equality~\eqref{equ_cylindrical_wiener_process_theta}. In particular, the real-valued process $\W_t(e_k)$, $t\geq 0$, satisfies
\begin{align*}
\W_{t}(e_k) &=  (\Y_{t},e_k)_{L_2} + \int_{ 0 }^{ t } \I_{\left\{ s\geq \tau_k \right\}}e_k \cdot \mathrm dB_s,
\end{align*}
because $\pr_{\Y_s}^\bot e_k =  \I_{\left\{ s  \geq \tau_k \right\}} e_k$. By construction of $e_k$ in Lemma~\ref{lem:onb}, $(\Y_{t},e_k)_{L_2}$ vanishes for all $t \geq \tau_k$. Thus, we note that for $t \in [0, \tau_k]$, $\W_{t}(e_k)=(\Y_{t},e_k)_{L_2}$ and that $\W_{\tau_k}(e_k)=0$, whereas for $t \geq \tau_k$, $\W_{t}(e_k)= B_t(e_k) - B_{\tau_k}(e_k)$. 
Since $B$ is independent of $\Y$ and thus of $e_k$,  $B_t(e_k)$ is well-defined by $  B_t(e_k)=\int_{ 0 }^{ t }e_k \cdot  \mathrm dB_s$, $t\geq 0$. 
To recap, in space direction $e_k$, the projection of $\W$ is equal to the projection of its coalescing part $\Y$ before stopping time $\tau_k$, and is equal to the projection of a noise $B$ which is independent of $\Y$ after $\tau_k$. Therefore, we define formally  $\xi= \T(\X)=\T(\Y,\W)$  as follows
\begin{align*} 
  \xi_t = \sum_{k=1}^\infty  e_k  \W_{t+\tau_k}(e_k) , \quad t\geq 0.
\end{align*}
More rigorously\footnote{Similarly as for the cylindrical Wiener process $\W$, $\xi$ can not be defined as a random process taking values in $L_2$.}, 
we define $\xi_t$ as a map from the Hilbert space $\ltz:=L_2\ominus \spann\{\I_{[0,1]}\}$ to $L_2(\Omega)$. We set
\begin{align}
  \label{equ_xi}
  \xi_t(h) :=  \sum_{k=1}^\infty  (e_k,h)_{L_2} \W_{t+\tau_k}(e_k), \quad t\geq 0, \ \ h \in \ltz.
\end{align}

\begin{proposition}
\label{prop:xi_is_cylindricalWp}
For every $h \in \ltz$ the sum~\eqref{equ_xi} converges almost surely in $\mC[0,\infty)$. Moreover, $\xi_t$, $t \geq 0$, is a cylindrical Wiener process in $\ltz$ starting at $0$ that is independent of the MMAF $\Y$. 
\end{proposition}

In order to prove the above statement, we start with the following lemma. 

\begin{lemma}
\label{lem:shifted_W_is_BM}
The processes $\W_{\cdot+\tau_k}(e_k)$, $k\geq 1$, are independent standard Brownian motions that do not depend on the MMAF $\Y$. 
\end{lemma}

\begin{proof}
Let us denote
\begin{align}
\label{def_eta}
\eta_k(t):=\W_{t+\tau_k}(e_k)= B_{t+\tau_k}(e_k) - B_{\tau_k}(e_k), \quad t \geq 0, \ \ k \geq 1.
\end{align}
  We fix $n\geq 1$ and show that the processes $\Y$, $\eta_k$, $k \in [n]$, are independent and that $\eta_k$, $k \in [n]$, are standard Brownian motions. Let 
  \[
    F_0: \mC([0,\infty),\Li) \to \R, \quad F_k: \mC[0,\infty) \to \R, \quad k \in [n],
  \]
  be bounded measurable functions. By strong Markov property of $B$ and the independence of $B$ and $\Y$, $B_{\cdot +\tau_k}-B_{\tau_k}$ is also independent of $\Y$. Moreover for every $y \in \coal$,
  \[
    \eta_k^y(t):=B_{t+\tau_k^y}(e_k^y) - B_{\tau_k^y}(e_k^y), \quad t\geq 0, \ \ k \in [n],
  \]
  are independent standard Brownian motions. Therefore, we can compute
\begin{align*}
  \E{ F_0\left( \Y \right)\prod_{ k=1 }^n F_k\left( \eta_k \right)}
  &=  \E{ \E{ F_0\left( \Y \right)\prod_{ k=1 }^{ n } F_k\left( \eta_k \right)\bigg| \Y } } \\
  &=\E{F_0\left( \Y \right)  \E{\prod_{ k=1 }^{ n } F_k\left( \eta_k^y \right)}\bigg|_{y=\Y} }\\
  &=\E{ F_0\left( \Y \right)\E{\prod_{ k=1 }^{ n } F_k\left( w_k \right)}\bigg|_{y=\Y} }\\
  &= \E{ F_0\left( \Y \right) }\prod_{ k=1 }^{ n } \E{ F_k(w_k) },
\end{align*}
where $w_k$, $k \in [n]$, are independent standard Brownian motions that do not depend on $\Y$. This completes the proof of the lemma.
\qed
\end{proof}

\begin{proof}[Proof of Proposition~\ref{prop:xi_is_cylindricalWp}]
Let $h \in \ltz$  and $y \in \coal$ be fixed. For every $n \in \N$ we define
\begin{align*}
  M_t^{y, n} (h) :=  \sum_{k=1}^n  (e_k^y, h)_{L_2}\eta_k(t), \quad t\geq 0,
\end{align*}
where $\eta_k$, $k \geq 1$, are defined by~\eqref{def_eta}. 
By Lemma~\ref{lem:shifted_W_is_BM}, $\eta_k$, $k\geq 1$, are independent standard Brownian motions,  hence $M_t^{y, n} (h)$ , $t\geq 0$, is a continuous square-integrable martingale with respect to the filtration $(\F^{\eta}_t)_{t\geq 0}$ generated by $\eta_k$, $k\geq 1$, with quadratic variation $
  \langle M^{y, n} (h) \rangle_t = \sum_{k=1}^n  (e_k^y, h)_{L_2}^2t$, $t\geq 0$.
Moreover,  for each $T>0$ the sequence of processes $\{M^{y,n}(h)\}_{n \geq 1}$  restricted to the interval $[0,T]$  converges in $L_2(\Omega, \mC [0,T])$. Indeed, for each $m < n$, by Doob's inequality
\begin{align*}
  \E{\sup_{t \in [0,T]} \left|M^{y,n}_t(h)-M^{y,m}_t(h) \right|^2 }
  &= \E{\sup_{t \in [0,T]} \left|\sum_{k=m+1}^n  (e_k^y, h)_{L_2} \eta_k(t) \right|^2  } \\
&\leq  4\sum_{k=m+1}^n (e_k^y, h)_{L_2}^2 T,
\end{align*}
The sum $\sum_{k=1}^n (e_k^y, h)_{L_2}^2$ converges  to $\|h\|_{L_2}^2$ because $\{e_k^y,\ k \geq 1\}$ is  an orthonormal basis of $\ltz$.  Thus, $\{M^{y,n}(h)\}_{n \geq 1}$ is a Cauchy sequence in the space $L_2(\Omega, \mC [0,T])$, and hence, it converges to a limit denoted by $M^y(h)= \sum_{k =1}^\infty (e_k^y, h)_{L_2} \eta_k$.  Trivially, $M^y_t(h)$ can be well-defined for all $t\geq 0$, and, by \cite[Lemma~B.11]{Cherny:2005}, $M^y_t(h)$, $t \geq 0$, is a continuous square-integrable $(\F^{\eta}_t)$-martingale with quadratic variation $\langle M^y(h) \rangle_t= \lim_{n \to \infty} \langle M^{y, n} (h) \rangle_t= \|h\|^2_{L_2} t$, $t\geq 0$. 

Remark that $\sum_{k=1}^{\infty} (e_k^y, h)_{L_2} \eta_k$ is a sum of independent random elements in $\mC [0, T]$. Thus, by It\^o-Nisio's Theorem~\cite[Theorem~3.1]{Ito:1968},  $\{M^{y, n} (h)\}_{n\geq 1}$ converges almost surely to $M^y(h)$ in $\mC [0, T]$ for every $T>0$, and therefore, in $\mC[0,\infty)$.  Recall that by Lemma~\ref{lem:shifted_W_is_BM}, the sequence $\{\eta_k\}_{k\geq 1}$ is independent of $\Y$, and by Lemma~\ref{lem:mmaf_belongs_to_coal}, $\Y$ belongs to $\coal$ almost surely. Then $\sum_{k=1}^{\infty}  (e_k, h)_{L_2}\eta_k$ also converges almost surely in $\mC[0, \infty)$ to a limit that we  have called $\xi(h)$. 

Moreover, similarly as the proof of Lemma~\ref{lem:shifted_W_is_BM}, we show that the processes $\Y$ and $\{\xi(h_i),\ i \in [n]\}$ for every $h_i \in \ltz$, $i \in [n]$, $n\geq 1$, are independent. We conclude that $\xi$ is independent of $\Y$. 

Let us show that $\xi$ is a cylindrical Wiener process. Obviously, $h \mapsto \xi(h)$ is a linear map. We denote $\tilde{\F}^{\eta,\Y}_t=\F^{\eta}_t\vee\sigma(\Y)$, $t\geq 0$. We need to check that for every $h \in \ltz$, $\xi(h)$ is an $(\tilde{\F}^{\eta,\Y}_t)$-Brownian motion. According to L\'evy's characterization of Brownian motion \cite[Theorem~II.6.1]{Ikeda:1989}, it is enough to show that $\xi(h)$ is a continuous square-integrable $(\tilde{\F}^{\eta,\Y}_t)$-martingale with quadratic variation $\|h\|_{L_2}^2t$. So, we take $n\geq 1$ and a bounded measurable function
\[
  F: \mC[0,\infty)^{n}\times \mC([0,\infty),L_2) \to \R.
\]
Then using Lemma~\ref{lem:shifted_W_is_BM} and the fact that $M^y(h)$ is an $(\F^{\eta}_t)$-martingale, we have for every $s<t$
\begin{align*}
  \e[ \xi_t(h)F\left( \left(\eta_k(\cdot \wedge s)\right)_{k=1}^n,\Y \right) ]&= \E{ \E{ \xi_t(h)F\left( \left(\eta_k(\cdot \wedge s)\right)_{k=1}^n,\Y \right)\big|\Y } }\\
  &= \E{ \E{ M^y_t(h)F\left( \left(\eta_k(\cdot \wedge s)\right)_{k=1}^n,y \right)}\Big|_{y=\Y}  }\\
  &= \E{ \E{ M^y_s(h)F\left(  \left(\eta_k(\cdot \wedge s)\right)_{k=1}^n,y \right)}\Big|_{y=\Y}  }  \\
&= \E{ \xi_s(h)F\left( \left(\eta_k(\cdot \wedge s)\right)_{k=1}^n,\Y \right) }.
\end{align*}
Hence, $\xi(h)$ is an $(\tilde{\F}^{\eta,\Y}_t)$-martingale. Similarly, one can prove that $\xi_t(h)^2-\|h\|_{L_2}^2t$, $t\geq 0$, is also an $(\tilde{\F}^{\eta,\Y}_t)$-martingale. This proves that $\xi(h)$ is a continuous square-integrable $(\tilde{\F}^{\eta,\Y}_t)$-martingale with quadratic variation $\|h\|_{L_2}^2t$, $t\geq 0$. The equality $\E{ \xi_t(h_1)\xi_t(h_2) }=t ( h_1,h_2 )_{L_2}$, $t\geq 0$, trivially follows from the polarization equality and the fact that $\xi(h_1)$ and $\xi(h_2)$ are martingales with respect to the same filtration $(\tilde{\F}^{\eta,\Y}_t)_{t\geq 0}$. Thus, $\xi$ is an $(\tilde{\F}^{\eta,\Y}_t)$-cylindrical Wiener process in $\ltz$ starting at $0$. This finishes the proof of the proposition.
\qed
\end{proof}

We conclude this section by defining properly the space $\bE$ on which the random element $\X$ take values and the non-coalescing remainder map $\T: \bE \to \bF$ needed to achieve step (S3) from the introduction.
However, as  we already noted, the cylindrical Wiener process $\W$ is not a random element in $\mC([0,\infty),L_2)$. 
So we define $\bE:=\mC([0,\infty),\Li)\times \mC[0,\infty)^{\N_0}$ and $\bF:=\mC_0[0,\infty)^{\N}$. 
Here, $\mC[0,\infty)$ is the space of continuous functions from $[0, \infty)$ to $\R$ equipped with its usual Fréchet distance, $\mC_0[0,\infty)$ denotes the subspace of all functions vanishing at $0$ and $\N_0:= \N \cup \{0\}$. 
Equipped with the metric induced by the product topology, $\bE$ is a Polish space. 

Now, we fix an orthonormal basis $\{h_j,\ j\geq 0\}$ of $L_2$ such that $h_0=\I_{[0,1]}$.
In particular, $\{h_j,\ j\geq 1\}$  is an orthonormal basis of $\ltz$.
We identify the cylindrical Wiener process $\W$ with the following random element in $\mC[0,\infty)^{\N_0}$:
\[
  \widehat{\W}_t=\left(\widehat{\W}_j(t)\right)_{j\geq 0}:=\left(\W_t(h_j)\right)_{j\geq 0}, \quad  t\geq 0.
\]
Indeed $\W$ and $\widehat{\W}$ are related by $\W_t(h)=\sum_{ j=0 }^{ \infty } \widehat{\W}_j(t) ( h , h_j )_{L_2}$, for all $t\geq 0$ and $h \in L_2$, where the series converges in $\mC[0,\infty)$ almost surely for every $h \in L_2$. 

Similarly, we identify $\xi$ with $\widehat{\xi}_t=\left(\widehat{\xi}_j(t)\right)_{j\geq 1}:=\left(\xi_t(h_j)\right)_{j\geq 1}$, $t\geq 0$, and $\Y$ with $\widehat{\Y}_t=\left( \widehat{\Y}_j(t) \right)_{j\geq 0}:=\left( ( \Y_t,h_j )_{L_2} \right)_{j\geq 0}$, $t\geq 0$. 
By equality~\eqref{equ_xi}, $\widehat{\xi}$ and $\widehat{\W}$ are related by
\begin{equation} 
  \label{equ_map_for_widehat_t}
  \widehat{\xi}_j(t)=\sum_{ k=1 }^{ \infty }\sum_{ i=1 }^{ \infty } ( e_k,h_j )_{L_2}( e_k,h_i )_{L_2}\widehat{\W}_i(t+\tau_{k}), \quad t\geq 0,\ \ j\geq 1.
\end{equation}

We define $\widehat{\X}=\left( \Y,\widehat{\W} \right)$, which is a random element on in $\bE$. By~\eqref{equ_map_for_widehat_t}, there exists a measurable map $\widehat{\T}: \bE \to \bF$ such that 
\begin{align}
\label{hat_equality}
\widehat{\xi}=\widehat{\T}(\widehat{\X})
\end{align}
almost surely.

\subsection{Statement of the main result}
\label{parag:statement_main_result}

Let us clarify step (S4) from the introduction. According to Definition~\ref{def:value_along_xi_n}, we need to define a random sequence $\{\xi^n\}_{n \geq 1}$ in $\bF=\mC_0[0,\infty)^{\N}$ converging to $0$ in distribution and such that $\p^{\xi^n}$ is absolutely continuous with respect to the law of $\widehat{\T}(\widehat{\X})$. By~\eqref{hat_equality} and Proposition~\ref{prop:xi_is_cylindricalWp},  $\p^{\widehat{\T}(\widehat{\X})}$ is the law of a sequence of independent Brownian motions.

Let for each $n \geq 1$, $\xi^n:=(\xi^n_j)_{j \geq 1}$ be the sequence of  Ornstein-Uhlenbeck processes, independent of $\Y$,  that are strong solutions to the equations
\begin{equation}
\label{def_OU_process}
\begin{cases}
  \mathrm d \xi^n_j(t)= - \alpha_j^n \I_{\left\{ t\leq n \right\}} \xi^n_j(t) \mathrm dt + \mathrm d \widehat{\xi}_j(t), \\
\xi^n_j(0)=0,
\end{cases}
\end{equation}
where $\{\alpha_j^n,\ n,j \geq 1\}$ is a family of non-negative real numbers such that
\begin{enumerate}
  \item[(O1)] for every $n\geq 1$ the series $\sum_{ j=1 }^{ \infty } (\alpha_j^n)^2< +\infty$;

  \item[(O2)] for every $j\geq 1$, $\alpha_j^n \to +\infty$ as $n\to\infty$.   
\end{enumerate}

\begin{remark} 
  \label{rem_conditions_a_and_b}
  \begin{enumerate} 
    \item [(i)] Using Kakutani's theorem~\cite[p. 218]{Kakutani:1948} and Jensen's inequality, it is easily seen that Condition (O1) guaranties the absolute continuity of $\p^{\xi^n}$ with respect to $\p^{\widehat{\xi}}$ on $\mCo$. The indicator function in the drift is important, otherwise the law is singular. 
    Hence, Assumption (B1) of Definition~\ref{def:value_along_xi_n} is satisfied by the sequence $\{ \xi^n \}_{n\geq 1}$.

    \item [(ii)] Condition (O2) yields the convergence in distribution of $\{\xi^n\}_{n\geq 1}$ to $0$ in $\mCo$ (see Lemma~\ref{lem_convergence_of_ou_processes} below). Thus Assumption (B2) is also satisfied.
  \end{enumerate}
\end{remark}

The following theorem is the main result of the paper.

\begin{theorem}
\label{theo:infinite_dim}
The value of the conditional distribution of $\widehat{\X}=(\Y,\widehat{\W})$ to the event $\{ \widehat{\T}(\widehat{\X})=0\}$ along $\{\xi^n\}$ is the law of $(\Y,\widehat{\Y})$.
\end{theorem}

The event $\{ \widehat{\T}(\widehat{\X})=0\}$, which equals to $\{ \widehat{\xi}=0\}$, is by construction the event where the non-coalescing part of $\widehat{\W}$ vanishes.

\begin{remark} 
  \label{rem_non_strictly_increasing_function}
  For simplicity, we assumed in sections~\ref{parag:construction_map_T} and~\ref{parag:statement_main_result} that the initial condition $g$ is strictly increasing. Actually, everything remains true if $g$ is an arbitrary element of $\ltp$, up to replacing the space $L_2$ by the space $L_2(g)$. In particular, if $g$ is a step function, then $L_2(g)$ has finite dimension, equal to $N(g)$, and the orthonormal basis constructed in Lemma~\ref{lem:onb} and the sum in the definition of $\widehat{\xi}$ consists of finitely many summands. 
  \end{remark}

\section{Proof of the main theorem}
\label{sec:proof_main_theorem}

In order to prove Theorem~\ref{theo:infinite_dim}, we follow the strategy introduced in 
Section~\ref{sec:conditional_distribution}. We start by the construction of a quadruple $(\bG,\Psi,Y,Z)$ satisfying (P1)-(P4). The idea behind the construction of $\Psi$ is inspired by the result of Proposition~\ref{prop_mmaf_and_cylindrical_brownian_motion}, stating that $\W$ can be build from the MMAF $\Y$ and  some independent process.

\subsection{Construction of quadruple}
\label{parag:quadruple}

Define $\bG:=\coal$, $Y:=\Y$ and $Z:=\widehat{\mZ}$, where $\mZ$ is a cylindrical Wiener process in $\ltz$ starting at $0$ that is independent of $\Y$. By the same identification as previously, for the same basis $\{h_j,\ j\geq 0\}$,  $\widehat{\mZ}_t=\left(\widehat{\mZ}_j(t)\right)_{j\geq 1}:=\left(\mZ_t(h_j)\right)_{j\geq 1}$, $t \geq 0$,  is a sequence of independent standard Brownian motions and is a random element in $\bF$. Therefore, properties (P1) and (P2) are satisfied. 

We define
\[
  \psi(\Y,\mZ):= \left(\Y, \varphi(\Y,\mZ)\right),
\]
where  $\varphi_t(\Y,\mZ)$ is a map from $L_2$ to $L_2(\Omega)$ defined by 
\begin{align}
\label{formula_phi}
\varphi_t(\Y,\mZ) (h)= (\Y_t,h)_{L_2} + \sum_{k=1}^\infty (e_k,h)_{L_2} \I_{\left\{ t \geq \tau_k \right\}} \mZ_{t-\tau_k}(e_k) 
\end{align}
for all $t\geq 0$ and $h \in L_2$. 
As in the proof of Lemma~\ref{lem:shifted_W_is_BM}, one can show that $\mZ(e_k)$, $k\geq 1$, are independent standard Brownian motions that do not depend on $\Y$.

\begin{lemma}
\label{lem:phi_is_cylindricalWp}
For each $h \in L_2$, the sum in~\eqref{formula_phi} converges almost surely in $\mC[0, \infty)$. Furthermore, 
$\varphi(\Y,\mZ)$ is a cylindrical Wiener process in $L_2$ starting at $g$ and the law of $\psi(\Y,\mZ)$ is equal to the law of $\X=(\Y,\W)$. 
\end{lemma}

\begin{remark}
\label{explanation_phi}
Before giving the proof of the lemma, remark that the map $\varphi$  constructs a cylindrical Wiener process from $\Y$, by adding to $\Y$ some non-coalescing term. Actually,  for each $y \in \coal$, $\varphi(y,z)$ belongs to $\coal$ if and only if $z=0$. This statement is proved in Lemma~\ref{lem:set_coal_as_preimage}.
\end{remark}

\begin{proof}[Proof of Lemma~\ref{lem:phi_is_cylindricalWp}]
  Let us first show that the sum in~\eqref{formula_phi} converges almost surely in $\mC[0,\infty)$. Fixing $y \in \coal$ and $h \in L_2$, we define for every $n \geq 1$
\begin{align*}
R^{y,n}_t (h):= \sum_{k=1}^n  (e_k^y,h)_{L_2} \I_{\left\{ t \geq \tau_k^y \right\}}\mZ_{t- \tau_k^y}\left(e_k\right), \quad t\geq 0. 
\end{align*}
Since $\mZ(e_k)$, $k\geq 1$, are independent standard Brownian motions, one can easily check that $R^{y,n}_t (h)$, $t\geq 0$, is  a continuous square-integrable martingale with respect to the filtration generated by $\mZ_{t-\tau_k^y}(e_k)$, $k\geq 1$. As in the proof of Proposition~\ref{prop:xi_is_cylindricalWp}, one can show that the sequence of partial sums $\{R^{y,n}(h)\}_{n\geq 1}$ converges in $\mC[0,\infty)$ almost surely for each $y \in \coal$. By the independence of $\mZ(e_k)$, $k\geq 1$, and $\Y$, one can see that the series 
\[
  R^{\Y}_t(h):=\sum_{k=1}^\infty (e_k,h)_{L_2} \I_{\left\{ t \geq \tau_k \right\}} \mZ_{t-\tau_k}(e_k), \quad t\geq 0,
\]
also converges almost surely in $\mC[0,\infty)$.

Next, we claim that there exists a cylindrical Wiener process $\theta_t$, $t \geq 0$, in $\ltz$ starting at $0$ independent of $\Y$ such that
\begin{equation} 
  \label{equ_equation_with_cwp_in_l0}
  \W_t=\Y_t+\int_0^t\pr_{\Y_s}^\bot  \mathrm d \theta_s,\quad t\geq 0.
\end{equation}
Indeed, by Proposition~\ref{prop_mmaf_and_cylindrical_brownian_motion}, there is a cylindrical Wiener process $B_t$, $t \geq 0$, in $L_2$ starting at $0$ independent of $\Y$ and satisfying equation~\eqref{equ_cylindrical_wiener_process_theta}. Taking $\theta$ equal to the restriction of $B$ to the sub-Hilbert space $\ltz$, we easily check that $\int_0^t\pr_{\Y_s}^\bot  \mathrm d \theta_s=\int_0^t\pr_{\Y_s}^\bot  \mathrm d B_s$, $t\geq 0$, since for all $s \geq 0$, $\pr_{\Y_s}^\bot=\pr_{\ltz}\circ \pr_{\Y_s}^\bot$ almost surely. Furthermore, almost surely
\begin{align*}
  \int_0^t\pr_{\Y_s}^\bot  \mathrm d \theta_s = \sum_{k=1}^\infty  e_k \I_{\left\{ t \geq \tau_k \right\}} (\theta_{t \wedge \tau_k}(e_k) - \theta_{\tau_k} ( e_k )), \quad t\geq 0.
\end{align*}
For each fixed $y \in \coal$, the family
  \[ 
    \left\{\I_{\left\{ t\geq \tau_{k}^y \right\}}(\theta_{t \wedge \tau_k^y}(e_k^y) - \theta_{\tau_k^y} ( e_k^y )),\ t \geq 0,\ k \geq 1 \right\},
  \]
  has the same distribution as 
  \[
    \left\{ \I_{\left\{ t \geq \tau_{k}^y \right\}}\mZ_{t- \tau_k^y} ( e_k^y),\ t \geq 0,\  k \geq 1  \right\}.
  \]
Therefore, using the independence of $\Y$ and $\theta$ on the one hand and the independence of $\Y$ and $\mZ$ on the other hand, we get the equality
\begin{multline*}
  \law \left\{ \left(\Y_t, \int_0^t\pr_{\Y_s}^\bot  \mathrm d \theta_s \right) , t\geq 0 \right\} \\= \law \left\{ \left(\Y_t, \sum_{k=1}^\infty  e_k \I_{\left\{ t \geq \tau_k \right\}} \mZ_{t- \tau_k} ( e_k) \right) , t\geq 0 \right\}.
\end{multline*}
This relation and equalities~\eqref{formula_phi} and~\eqref{equ_equation_with_cwp_in_l0} yield that the law of $\X=(\Y, \W)$ is equal to the law of $\psi(\Y, \mZ)=(\Y, \varphi(\Y, \mZ))$. In particular, $\varphi(\Y, \mZ)$ is a cylindrical Wiener process in $L_2$ starting at $g$. 
\qed
\end{proof}

Moreover, there exists a measurable map $\widehat{\varphi}: \bE \to \mC[0,\infty)^{\N_0}$ such that 
\[
  \widehat{\varphi}(\Y,\widehat{\mZ})=\reallywidehat{\varphi(\Y,\mZ)}.
\]
almost surely. Let us define  $\Psi:\bG \times \bF \to \bE$ by
\begin{equation}
\label{def:psi}
\Psi(y,z):=\left(y,\widehat{\varphi}(y,z)\right).
\end{equation}
It follows from the last two equalities and from Lemma~\ref{lem:phi_is_cylindricalWp} that
\begin{corollary}
The laws of $\Psi(\Y,\widehat{\mZ})$ and of $\widehat{\X}=(\Y,\widehat{\W})$ are the same. 
\end{corollary}
Hence Property (P4) is satisfied. It remains to check (P3). 
By equalities~\eqref{equ_map_for_widehat_t} and~\eqref{hat_equality}, we  compute $\widehat{\T}(\Psi(\Y,\widehat{\mZ}))$:
\[
  \widehat{\T}(\Psi(\Y,\widehat{\mZ}))_j(t)=\sum_{ k=1 }^{ \infty }\sum_{ i=1 }^{ \infty } ( e_k,h_j )_{L_2}( e_k,h_i )_{L_2}\reallywidehat{\varphi(\Y,\mZ)}_i(t+\tau_{k}), \quad t\geq 0,\ \ j\geq 1.
\]

\begin{proposition}
\label{prop:T_of_psi_is_identity}
Almost surely $\widehat{\T}(\Psi(\Y,\widehat{\mZ}))=\widehat{\mZ}$. 
\end{proposition}

\begin{proof}
  By  continuity in $t$ of $\widehat{\T}(\Psi(\Y,\widehat{\mZ}))_j(t)$ and $\widehat{\mZ}_j(t)$,  it is enough to show that for each $t \geq 0$ and $j \geq 1$ almost surely $\widehat{\T}(\Psi(\Y,\widehat{\mZ}))_j(t)=\widehat{\mZ}_j(t)$. 
Since $\{h_i,\ i\geq 1\}$ is an orthonormal basis of $\ltz$, we have
\begin{align*}
\widehat{\T}(\Psi(\Y,\widehat{\mZ}))_j(t)
&=\sum_{ k=1 }^{ \infty }\sum_{ i=1 }^{ \infty } ( e_k,h_j )_{L_2}( e_k,h_i )_{L_2}\varphi_{t+\tau_{k}}(\Y,\mZ) (h_i) \\
&=\sum_{ k=1 }^{ \infty } ( e_k,h_j )_{L_2}\varphi_{t+\tau_{k}}(\Y,\mZ) (e_k).
\end{align*} 
By~\eqref{formula_phi}  and Lemma~\ref{lem:onb}, we have 
\begin{align*}
\varphi_{t+\tau_k}(\Y,\mZ)(e_k) &=  (\Y_{t+\tau_k},e_k)_{L_2} + \sum_{l=1}^\infty (e_l,e_k)_{L_2} \I_{\left\{ t+\tau_k \geq \tau_l\right\}} \mZ_{t+\tau_k- \tau_l} ( e_l)\\
&= \I_{\left\{ t+\tau_k \geq \tau_k\right\}} \mZ_{t+\tau_k- \tau_k} ( e_k)=\mZ_t(e_k).
\end{align*}
Hence, almost surely
\begin{align*}
\widehat{\T}(\Psi(\Y,\widehat{\mZ}))_j(t)=\sum_{k=1}^\infty  (e_k, h_j)_{L_2}\mZ_t \left( e_k \right)= \mZ_t(h_j)=\widehat{\mZ}_j(t),
\end{align*}
because $\{e_k,\ k \geq 1\}$ is an orthonormal basis of $\ltz$.
\qed
\qed
\end{proof}

Thus, Property (P3) holds.  Hence, by Proposition~\ref{prop:existence_regular}, the probability kernel $p$ defined by 
\begin{equation}
\label{defin:p}
  p(A,z):=  \P{\Psi(\Y,z) \in A}=\P{ \left( \Y,\widehat{\varphi}\left( \Y,z \right)\right)\in A }
\end{equation}
for all $A \in \mathcal B(\bE)$ and $z \in \bF$, is a regular conditional probability of $\widehat{\X}$ given $\widehat{\T}(\widehat{\X})$.

\subsection{Value of $p$ along a sequence of Ornstein-Uhlenbeck processes}
\label{parag:value_of_regular_conditional_probability_along_a_sequence}

According to Proposition~\ref{prop:existence_regular}, it remains to show the following to complete the proof of Theorem~\ref{theo:infinite_dim}. Let $\left\{\xi^n\right\}_{n\geq 1}$ be the sequence defined by~\eqref{def_OU_process} and 
 independent of $\Y$. Let $\Psi$ be defined by~\eqref{def:psi}. Then $\Psi(\Y, \xi^n)$ converges in distribution to $(\Y,\widehat{\Y})$.

For $y \in \coal$ we consider 
\[
  \Psi(y,\xi^n)=(y,\widehat{\varphi}(y,\xi^n)),
\]
where the map $\widehat{\varphi}:\bE \to \mC[0,\infty)^{\N_0}$ was defined in Section~\ref{parag:quadruple}. Since for every $n\geq 1$ the law of $\xi^n$ is absolutely continuous with respect to  $\p^{\widehat{\xi}}$ (which is equal to $\p^{\widehat{\mZ}}$), we have that for almost all $y \in \coal$ with respect to  $\p^\Y$
\begin{equation} 
  \label{equ_equality_for_phi_of_xi}
  \widehat{\varphi}_j\left(y,\xi^n\right)=(y_{\cdot },h_j)+\sum_{ k=1 }^{ \infty } \sum_{ l=1 }^{ \infty } ( e_k^y,h_j )_{L_2}( h_l,e_k^y )_{L_2}\I_{\left\{ \cdot \geq \tau_k^y \right\}}\xi_l^n(\cdot -\tau_k^y)
\end{equation}
for each $j\geq 0$, where the series converges in $\mC[0,\infty)$ almost surely. Without loss of generality, we may assume that equality~\eqref{equ_equality_for_phi_of_xi} holds for all $y \in \coal$. Otherwise, we can work with a measurable subset of $\coal$ of $\p^\Y$-measure one  for which equality~\eqref{equ_equality_for_phi_of_xi} holds.

\begin{proposition}
\label{prop:convergence_along_xi_n}
Let $\eps \in (0,1)$ and $y \in \coal$ be such that $\sum_{k=1}^\infty (\tau_k^y)^{1-\eps}< \infty$.  Then the  sequence of processes $\Psi(y,\xi^n)$, $n \geq 1$, converges in distribution to $(y,\widehat{y})$ in $\bE=\mC ([0, \infty), \Li) \times \mCz$, where $\widehat{y}=(( y_{\cdot },h_j )_{L_2})_{j\geq 0}$. 
\end{proposition}

Let us fix $y \in \coal$ satisfying the assumption of Proposition~\ref{prop:convergence_along_xi_n}. Before starting the proof, we define for all $j\geq 0$
\begin{align*}
  R^n_j(t):= \sum_{ k=1 }^{ \infty } \sum_{ l=1 }^{ \infty } ( e_k^y,h_j )_{L_2}( h_l,e_k^y )_{L_2}\I_{\left\{ t\geq \tau_k^y \right\}}\xi_l^n(t -\tau_k^y), \quad t\geq 0,
\end{align*}
and $R_t^n:=(R^n_j(t))_{j\geq 0}$, $t\geq 0$.  Remark that $R^n_0=0$. 
Note that it is sufficient to prove that 
\begin{align}
\label{conv_Rn}
R^n \stackrel{d}{\to } 0 \quad  \mbox{in} \ \ \mCz, \ \ n \to \infty. 
\end{align}
Indeed, this will imply that 
\begin{align*}
  \Psi(y,\xi^n)=\left(y,\widehat{\varphi}(y,\xi^n)\right)=\left( y,\widehat{y}+R^n \right) \stackrel{d}{\to } (y,\widehat{y}) \quad \mbox{in}\ \ \bE.
\end{align*}

Let us first prove some auxiliary lemmas.

\begin{lemma} 
  \label{lem_convergence_of_ou_processes}
  The sequence of random elements $\{ \xi^n \}_{n\geq 1}$ converges in distribution to $0$ in $\mCo$.   
\end{lemma}

\begin{proof} 
  In order to prove the lemma, we first show that the sequence $\{ \xi^n \}_{n\geq 1}$ is tight in $\mCo$. This will imply that the sequence $\{ \xi^n \}_{n\geq 1}$ is relatively compact, by Prohorov's theorem. Then we will show that every (weakly) convergent subsequence of $\{ \xi^n \}_{n\geq 1}$ converges to $0$. This will immediately yield that $\xi^n \stackrel{d}{\to }0$ in $\mCo$.

  According to \cite[Proposition~3.2.4]{Ethier:1986}, the tightness of $\{ \xi^n \}_{n\geq 1}$ will follow from the tightness of $\{ \xi^n_j \}_{n\geq 1}$ in $\mC[0,\infty)$ for every $j\geq 1$. So, let $j\geq 1$ and $T>0$ be fixed. Since the covariance of Ornstein-Uhlenbeck processes is well-known, one can easily check that for every $n \geq 1$ and every $0\leq s\leq t\leq n$, 
\begin{equation}
\label{ornst_uhl_2}
\E{\left(\xi_j^n(t) -\xi_j^n (s)  \right)^2} \leq \frac{1}{\alpha_j^n}\wedge (t-s),
\end{equation}
where $ \frac{1}{ 0 }:=+\infty$.
  Since $\xi^n_j$ is a Gaussian process, it follows that for every $0\leq s\leq t\leq T$ and every $n \geq T$, 
    \[
    \E{ \left( \xi^n_j(t)-\xi^n_j(s) \right)^4 }\leq 3\E{ \left( \xi^n_j(t)-\xi^n_j(s) \right)^2 }^2\leq 3(t-s)^2.
  \]
  Moreover, $\xi^n_j(0)=0$. Hence, by  Kolmogorov-Chentsov tightness criterion (see e.g.~\cite[Corollary 16.9]{Kallenberg:2002}), the sequence of processes $\{ \xi^n_j \}_{n\geq 1}$ restricted to $[0,T$] is tight in $\mC[0,T]$. Since $T>0$ was arbitrary, we get that $\{ \xi^n_j \}_{n\geq 1}$ is tight in $\mC[0,\infty)$. Hence, $\{ \xi^n \}_{n\geq 1}$ is tight in $\mCo$.
  
  Next, let $\{ \xi^n \}_{n\geq 1}$ converges in distribution to $\xi^\infty$ in $\mCo$ along a subsequence $N \subseteq \N$. Then for every $t\geq 0$ and $j\geq 1$ $\{ \xi^n_j(t) \}_{n\geq 1}$ converges in distribution to $\xi^\infty_j(t)$ in $\R$ along $N$. But on the other hand, for each $n \geq t$, 
\[
  \E{ (\xi^n_j(t))^2}\leq \frac{t}{ \alpha^n_j } \to 0,\quad  n\to\infty,
\]
by~\eqref{ornst_uhl_2} and Assumption (O2) in Section~\ref{parag:statement_main_result}. Hence, $\xi^\infty_j(t)=0$ almost surely for all $t\geq 0$ and $j\geq 1$. Thus, we have obtained that $\xi^\infty=0$, and therefore, $\xi^n \stackrel{d}{\to } 0$ in $\mCo$ as $n\to\infty$.
\qed
\end{proof}

To prove that $\{ R^n \}_{n\geq 1}$ converges to $0$, we will use the same argument as in the proof of Lemma~\ref{lem_convergence_of_ou_processes}. So, we start from the tightness of $\{ R^n \}$.

\begin{lemma}
\label{lem_tight1}
Under the assumption of Proposition~\ref{prop:convergence_along_xi_n}, the sequence $\{R^n\}_{n\geq 0}$ is tight in $\mCz$. 
\end{lemma}

\begin{proof}
  Again, according to \cite[Proposition~3.2.4]{Ethier:1986}, it is enough to check that the sequence $\{R^n_j\}_{n\geq 1}$ is tight in $\mC[0,\infty)$ for every $j\geq 0$. For $j=0$, $R^n_0=0$ so the result is obvious. 
  So, let $j\geq 1$ be fixed. We set
\[
R^{n,1}_j(t):= \sum_{ k=1 }^{ \infty } \sum_{ l=1 }^{ \infty } ( e_k^y,h_j )_{L_2}( h_l,e_k^y )_{L_2}\xi_l^n(t), \quad t\geq 0,
\]
and
\[
R^{n,2}_j(t):= \sum_{ k=1 }^{ \infty } \sum_{ l=1 }^{ \infty } ( e_k^y,h_j )_{L_2}( h_l,e_k^y )_{L_2}\left(\I_{\left\{ t\geq \tau_k^y \right\}}\xi_l^n(t -\tau_k^y)-\xi_l^n(t)\right), \quad t\geq 0.
\]
Then $R^n_j=R^{n,1}_j+R^{n,2}_j$. We will prove the tightness separately for $\{ R^{n,1}_j \}_{n\geq 1}$ and $\{ R^{n,2}_j \}_{n\geq 1}$.

\textbf{Tightness of $\{R^{n,1}_j\}_{n\geq 1}$.}
Using the fact that $\{ e_k^y,\ k\geq 1\}$ and $\{ h_l ,\ l\geq 1\}$ are bases of $\ltz$, a simple computation shows that almost surely
\[
  \Gamma_j(\widehat{\xi}):= \sum_{ k=1 }^{ \infty } \sum_{ l=1 }^{ \infty } ( e_k^y,h_j )_{L_2}( h_l,e_k^y )_{L_2}\widehat{\xi}_l=\widehat{\xi}_j.
\]
Due to the absolute continuity of the law of $\xi^n$ with respect to the law of $\widehat{\xi}$ and the equality $\Gamma_j(\xi^n)=R^{n,1}_j$, we get that $R^{n,1}_j=\xi_j^n$. Hence it follows from Lemma~\ref{lem_convergence_of_ou_processes} that $R^{n,1}_j$ converges in distribution to $0$ in $\mC[0,\infty)$. In particular, $\{ R^{n,1}_j \}_{n\geq 1}$ is tight in $\mC[0,\infty)$, according to Prohorov's theorem.

\textbf{Tightness of $\{R^{n,2}_j\}_{n\geq 1}$.} 

{\it Step I.}
For any $t\in [0,n]$ the vector 
\begin{align*}
V^n_t:=\sum_{k=1}^\infty    \sum_{l=1}^\infty e_k^y(e_k^y, h_l)_{L_2}  \left( \I_{\left\{ t \geq \tau_k^y \right\}}\xi^n_l(t - \tau_k^y) - \xi^n_l(t) \right)
\end{align*}
belongs almost surely to $\ltz$ and $\E{\| V^n_t \|_{L_2}^2} \leq \sum_{k=1}^\infty (t \wedge \tau_k^y) < \infty$. 

Indeed, by Parseval's equality (with respect to the  orthonormal family $\{e_k^y,\ k\geq 1\}$) and by the independence of $\{\xi^n_l\}_{l\geq 1}$, 
\begin{equation}
\label{parseval1}
\E{ \left\| V^n_t \right\|_{L_2}^2 }= \sum_{k=1}^\infty  \sum_{l=1}^\infty  (e_k^y, h_l)_{L_2}^2 E^n_{k,l}(t),
\end{equation}
where $E^n_{k,l}(t):=\E{\left( \I_{\left\{ t \geq \tau_k^y \right\}}\xi^n_l(t - \tau_k^y) - \xi^n_l(t) \right)^2}$.
Since $\xi^n_l(0)=0$, we have
\begin{equation}
\label{decomposition_En}
  E^n_{k,l}(t)=  \I_{\left\{ t \geq \tau_k^y \right\}} \E{\left( \xi^n_l(t - \tau_k^y) - \xi^n_l(t) \right)^2} +\I_{\left\{ t < \tau_k^y \right\}}\E{\left( \xi^n_l(0)-  \xi^n_l(t) \right)^2}.
\end{equation}
By inequality~\eqref{ornst_uhl_2}, we can deduce that 
\begin{equation}
\label{En_first_ineq}
E^n_{k,l}(t) \leq \I_{\left\{ t \geq \tau_k^y \right\}} \tau_k^y
+\I_{\left\{ t < \tau_k^y \right\}} t = t \wedge \tau_k^y.
\end{equation}
Therefore, 
\begin{equation} 
  \label{equ_estimate_for_v}
    \E{ \left\| V^n_t \right\|_{L_2} ^2}
    \leq  \sum_{k=1}^\infty  \sum_{l=1}^\infty  (e_k^y, h_l)_{L_2}^2 (t \wedge \tau_k^y)
    =  \sum_{k=1}^\infty (t \wedge \tau_k^y),
\end{equation}
by Parseval's identity (with respect to the  orthonormal family $\{h_l,\ l\geq 1\}$).
Moreover, $\sum_{k=1}^\infty (t \wedge \tau_k^y)
\leq t^\eps \sum_{k=1}^\infty ( \tau_k^y)^{1-\eps} < \infty$. 
 Therefore,  for any $t \in [0,n]$, $V^n_t$ belongs to $\ltz$ almost surely. 
 In particular, for every $t\in [0,n]$  the inner product $(V^n_t,h_j)_{L_2}$ is well-defined, and almost surely $R^{n,2}_j(t)=( V^n_t,h_j )_{L_2}$. 

 {\it Step II.}
 Let $T>0$. There exists $C_{y, \eps}$ depending on $y$ and $\eps$ such that for all $0\leq s \leq t\leq T$ and  $n \geq T$,  
\begin{align*}
  \E{\left(R^{n,2}_j(t) -R^{n,2}_j(s)\right)^2 } \leq C_{y, \eps} (t-s)^{\eps}.
\end{align*}

Indeed, proceeding as in Step~I, we get
\begin{multline*}
  \E{\left(R^{n,2}_j(t) -R^{n,2}_j(s)\right)^2 } \leq \E{ \left\| V^n_t -V^n_s \right\|_{L_2} ^2}  \\ 
  \begin{aligned}
&\leq \sum_{k=1}^\infty  \sum_{l=1}^\infty  (e_k^y, h_l)_{L_2}^2 \mathbb E \Big[\Big( \I_{\left\{ t \geq \tau_k^y \right\}}\xi^n_l(t - \tau_k^y) - \xi^n_l(t) \\
&\quad \quad \quad\quad\quad\quad\quad\quad- \I_{\left\{ s \geq \tau_k^y \right\}}\xi^n_l(s - \tau_k^y) + \xi^n_l(s) \Big)^2\Big] \\
&\leq \sum_{k=1}^\infty  \sum_{l=1}^\infty  (e_k^y, h_l)_{L_2}^2 4 \left( (t-s) \wedge \tau_k^y\right) 
= 4 \sum_{k=1}^\infty \left( (t-s) \wedge \tau_k^y\right) \\
&\leq 4 (t-s)^{\eps} \sum_{k=1}^\infty ( \tau_k^y )^{1-\eps},
\end{aligned}
\end{multline*}
where we use as previously  inequality~\eqref{ornst_uhl_2}. 
The series $\sum_{k=1}^\infty ( \tau_k^y )^{1-\eps}$ converges by assumption on $y$, so the proof of Step~II is achieved. 

{\it Step III.}
There exists $\alpha>0$, $\beta >0$ and $C_{y, \eps}$ depending on $y$ and $\eps$ such that for all $0\leq s \leq t\leq T$ and  $n \geq T$,  
\begin{align*}
  \E{\left|R^{n,2}_j(t) -R^{n,2}_j(s)\right|^\alpha } \leq C_{y,\eps} (t-s)^{1+\beta}.
\end{align*}

Indeed, for  any $s \leq t$ from $[0,T]$, $R^{n,2}_j(t)-R^{n,2}_j(s)$ is a random variable with normal distribution $\mathcal N (0, \sigma^2)$. By Step~II, $\sigma ^2 \leq C_{y, \eps} (t-s)^{\eps}$. Therefore, for any $p \geq 1$, 
\begin{align*}
\E{\left|R^{n,2}_j(t) -R^{n,2}_j(s)\right|^{2p} } \leq (2p-1) !! \  (\sigma^2)^p \leq C_{p,y, \eps} (t-s)^{\eps p}.
\end{align*}
The statement of Step~III follows by choosing $p$ larger than $\frac{1}{\eps}$. 

{\it Step IV.}
By Kolmogorov-Chentsov tightness criterion (see e.g.~\cite[Corollary 16.9]{Kallenberg:2002}), it follows from Step~III and the equality $R^{n,2}_j(0)=0$, $n \geq 1$, that the sequence of processes $\{R^{n,2}_j\}_{n\geq 1}$ restricted to $[0,T]$ is tight in $\mC[0,T]$ for every $T>0$. Hence, $\{ R^{n,2}_j \}_{n\geq 1}$ is tight in $\mC[0,\infty)$. 

\textbf{Conclusion of the proof.}
As the sum of two tight sequences, the sequence $\{R^n_j\}_{n\geq 1}$ is tight in $\mC[0, \infty)$ for any $j\geq 1$. Since $\mCo$ is equipped with the product topology, it follows from \cite[Proposition~3.2.4]{Ethier:1986} that the sequence $\{R^n\}_{n\geq 1}$ is tight in $\mCo$.
\qed
\end{proof}

\begin{lemma}
\label{lem_L2}
For every $j \geq 1$ and $t\geq 0$, $\E{(R^n_j(t))^2} \to 0$ as $n \to \infty$. 
\end{lemma}

\begin{proof}
Let $j \geq 1$ and $t\geq 0$ be fixed. We recall that $R^n_j=R^{n,1}_j + R^{n,2}_j$. Remark that $R^{n,1}_j= \xi^n_j$ almost surely. Thus, $\E{ \left(R^{n,1}_j(t)\right)^2 } \to 0$ follows immediately from inequality~\eqref{ornst_uhl_2}.

Due to the equality $R^{n,2}_j(t) = (V^n_t,h_j)_{L_2}$, we can estimate for $n\geq t$
\[
  \E{\left(R^{n,2}_j(t)\right)^2} \leq \E{\left\| V^n_t \right\|_{L_2}^2}=\sum_{k=1}^\infty  \sum_{l=1}^\infty  (e_k^y, h_l)_{L_2}^2 E^n_{k,l}(t).
\] 
By~\eqref{decomposition_En} and~\eqref{ornst_uhl_2}, we have for every $k,l\geq 1$
\[
  0\leq E^n_{k,l}(t) \leq \frac{1}{\alpha_l^n} \to 0, \quad n\to\infty.
\]
Therefore, inequalities~\eqref{En_first_ineq} and~\eqref{equ_estimate_for_v} and the dominated convergence theorem imply that $\E{\left\| V^n_t \right\|_{L_2}^2} \to 0$. This concludes the proof. 
\qed
\end{proof}

\begin{proof}[Proof of Proposition~\ref{prop:convergence_along_xi_n}]
  Lemma~\ref{lem_tight1} and Prohorov's theorem yield that the sequence $\{R^n\}_{n\geq 1}$  is relatively compact in $\mCz$. 
Moreover, by Lemma~\ref{lem_L2}, we deduce that each weakly convergent subsequence of $\{R^n\}_{n\geq 1}$ converges in distribution to $0$. It implies convergence~\eqref{conv_Rn}, which achives the proof of the proposition. 
 \qed
\end{proof}

\begin{proof}[Proof of Theorem~\ref{theo:infinite_dim}]
By lemmas~\ref{lem:mmaf_belongs_to_coal} and~\ref{lem:stopping_time_mmaf}, 
$\Y$ belongs almost surely to $\coal$ and the series $\sum_{k=1}^\infty (\tau_k^\Y)^{1-\eps}$ converges almost surely for each $\eps \in (0,\frac{1}{2})$.
Therefore, Proposition~\ref{prop:convergence_along_xi_n} and the independence of $\Y$ and  $\{\xi^n\}_{n \geq 1}$ imply that 
$\Psi(\Y,\xi^n)$, $n \geq 1$, converges in distribution to $(\Y,\widehat{\Y})$ in $\bE$. By Proposition~\ref{prop:existence_regular}, the same sequence converges in distribution to the conditional law $\law_{\{\xi^n\}}(\X | \T(\X)=0)$. Thus $\law_{\{\xi^n\}}(\X | \T(\X)=0)=\law(\Y,\widehat{\Y})$.
\qed
\end{proof}

\section{Coupling of MMAF and cylindrical Wiener process}
\label{sec:coupling}

We have already seen, in Proposition~\ref{prop_mmaf_and_cylindrical_brownian_motion} and its proof, that for every MMAF $\Y$ starting at $g$ there exists a cylindrical Wiener process $\W$  in $L_2$ starting at $g$ such that equation~\eqref{intro:equ_couple} holds. 
However, it is unknown whether equation~\eqref{intro:equ_couple} has a strong solution. 

In Proposition~\ref{prop_mmaf_and_cylindrical_brownian_motion}, we considered a process $\W$ defined by~\eqref{equ_cylindrical_wiener_process_theta} and we proved that the pair $(\Y, \W)$ satisfies~\eqref{intro:equ_couple}. The reverse statement holds true, in the following sense.

\begin{proposition}
\label{pro_uniqueness_of_coupling}
Let $\Y_t$, $t\geq 0$, be a MMAF  and $\W_t$, $t\geq 0$ be a cylindrical Wiener process in $L_2$ both starting at $g$ and such that $(\Y,\W)$ satisfies~\eqref{intro:equ_couple}. Then there exists a cylindrical Wiener process $B_t$, $t\geq 0$, in $L_2$ starting at $0$ independent of $(\Y,\W)$ such that for every $h \in L_2$ almost surely
\begin{align}
\label{equ_uniqueness_coupling}
\W_t(h)=(\Y_t,h)_{L_2}+\int_0^t\pr_{\Y_s}^\bot h \cdot \mathrm d B_s,\quad t\geq 0.
\end{align}
\end{proposition}

Proposition~\ref{pro_uniqueness_of_coupling} directly implies the statement of Theorem~\ref{theo:coupling}. 
Before we prove Proposition~\ref{pro_uniqueness_of_coupling}, we will show several auxiliary statements.

Recall that we denote $e_{k}:=e_k^\Y$ and $\tau_{k}:=\tau_k^\Y$, and that
for every $k\geq 1$, the random element $e_{k}$ is $\F^{\Y}_{\tau_k}$-measurable. 
Let $(\F_t^{\X})_{t\geq 0}$ be the complete right-continuous filtration generated by $\X:=(\Y,\W)$. 

For every $k\geq 1$ we remark that $\W^k_t:=\W_{t+\tau_k}-\W_{\tau_k}$, $t\geq 0$, is a cylindrical Wiener process starting at $0$ independent of $\F^{\X}_{\tau_k}$. Moreover, if $l \geq k$, then $\tau_l \leq \tau_k$ almost surely and the random element $e_l$ is $\F^{\X}_{\tau_l}$-measurable, hence  also $\F^\X_{\tau_k}$-measurable. Therefore,  the process 
\begin{equation}
\label{def:Wkel}
  \W^k_t(e_l):=\int_{ 0 }^{ t }  e_l \cdot \mathrm d\W^k_s =\int_{ \tau_k }^{ t+\tau_k }  e_l \cdot \mathrm d\W_s  , \quad t\geq 0,
\end{equation}
is well-defined. 

\begin{lemma} 
  \label{lem_independence_of_y_and_wk}
  The processes $\Y$, $\W^k(e_k)$, $k\geq 1$, are independent.
\end{lemma}

In order to prove that lemma, we start by some auxiliary definitions and results. 
The process
\[
  \zeta^k_t:=\int_{ 0 }^{ t } \pr_{\Y_{\tau_k}}\mathrm d \W^k_s , \quad t\geq 0,
\]
is a well-defined continuous $L_2$-valued $(\F^\X_{t+\tau_k})$-martingale, because $\pr_{\Y_{\tau_k}}$ is $\F^\X_{\tau_k}$-measurable and $\W^k_t$, $t \geq 0$, is independent of $\F^\X_{\tau_k}$.
Let $\G_k$ be the complete $\sigma$-algebra generated by $\X(t \wedge \tau_k)=(\Y_{t\wedge \tau_k},\W_{t\wedge \tau_k})$, $t\geq 0$, and by $\zeta_t^k$, $t\geq 0$.

\begin{lemma} 
  \label{lem_measurability_of_y}
  For every $k\geq 1$ the MMAF $\Y$ is $\G_k$-measurable as a map from $\Omega$ to $\mC([0,\infty),\Li)$.
\end{lemma}

\begin{proof} 
  In order to show the measurability of $\Y$ with respect to $\G_k$, it is enough to show the measurability of $\Y_{\tau_k+t}$, $t\geq 0$.
  
By Corollary~\ref{cor_existence_of_solution_to_the_equation}, we know that for   every  $g \in \St$ and cylindrical Wiener process $W$,  there exists a unique continuous $\Li$-valued process $Y$ such that almost surely
  \begin{equation*} 
    Y_t=g+\int_{ 0 }^{ t } \pr_{Y_s}\mathrm dW^g_s, \quad t\geq 0,
  \end{equation*}
  where $W_t^g=\int_{ 0 }^{ t } \pr_{g}\mathrm dW_s $, $t\geq 0$.

Let us consider the equation 
\begin{align}
\label{equ_Z_t}
Z_t=\Y_{\tau_k}+\int_{ 0 }^{ t } \pr_{Z_s}\mathrm d\zeta^k_s , \quad t\geq 0, 
\end{align}
where $\zeta^k_t=\int_{ 0 }^{ t } \pr_{\Y_{\tau_k}}\mathrm d \W^k_s$. 
We note that $\Y_{\tau_k}$ belongs to $\St$ almost surely and is independent of $\W^k$. Furthermore, the process $\Y_{\tau_k+t}$, $t\geq 0$, is a strong solution to~\eqref{equ_Z_t}. Therefore, it is uniquely determined by $\zeta^k$ and $\Y_{\tau_k}$, thus it is $\G_k$-measurable. 
\qed
\end{proof}

\begin{lemma} 
\label{lem_independence_for_wk}
Let $y \in \coal$ and $k \geq 1$.  Then the processes 
\[
\W^k_t(e_l^y)=\int_{ 0 }^{ t }  e_l^y \cdot \mathrm d\W^k_s , \quad t\geq 0, \ l \geq k
\]
are independent standard Brownian motions that do not depend on 
\[
\zeta^{y,k}_t:=\int_{ 0 }^{ t } \pr_{y_{\tau_k^y}}\mathrm d\W^k_s, \quad t\geq 0. 
\]
\end{lemma}

\begin{proof} 
By Lemma~\ref{lem:onb}, the family $\{e^y_l, l \geq 0\}$ is orthonormal.
Consequently, $\W^k(e^y_l)$, $l \geq 0$,  are independent Brownian motions.
Moreover, by Lemma~\ref{lem:onb} again, $\zeta^{y,k}_t= \sum_{j=0}^{k-1} e_j^k \W^k_t(e_j^y)$, $t \geq 0$, thus it is independent to $\W^k(e_l^y)$, $l \geq k$. 
\qed
\end{proof}

\begin{lemma} 
  \label{lem_properties_of_wk}
  For every $k\geq 1$ the processes $\W^k(e_l)$, $l\geq k$, are independent Brownian motions and do not depend on $\G_k$. Furthermore, for each $l>k$, $\W^l_{\cdot \wedge \tau_{k,l}}(e_l)$ is $\G_k$-measurable, where $\tau_{k,l}:=\tau_k-\tau_l$.
\end{lemma}

\begin{proof} 
Let $n\geq k$ and $m\geq 1$ be fixed. Let $h_j$, $j\geq 0$, be an arbitrary orthonormal basis of $L_2$.  We consider bounded measurable functions 
\begin{align*}
G_0: \mC([0,\infty),\Li) \times \mC[0,\infty)^m &\to \R \\
G_1: \mC([0,\infty),L_2) &\to \R \\
F_l:\mC[0,\infty) &\to \R, \quad l=k,\dots,n.
\end{align*}
We then use  the independence of $\W^k$ from $\F^\X_{\tau_k}$.
\begin{align*}
E&:= \E{ G_0\left(\Y_{\cdot \wedge \tau_k},\left( \W_{\cdot \wedge \tau_k}(h_{j}) \right)_{j=1}^m\right)
G_1(\zeta^k)
\prod_{ l=k }^{ n } F_l\left(\W^k(e_l)\right) } \\
&= \E{ G_0\left(\Y_{\cdot \wedge \tau_k},\left( \W_{\cdot \wedge \tau_k}(h_{j}) \right)_{j=1}^m\right) \E{
G_1(\zeta^k)
\prod_{ l=k }^{ n } F_l\left(\W^k(e_l)\right) \Bigg|\F^{\X}_{\tau_k} }} \\
&= \E{ G_0\left(\Y_{\cdot \wedge \tau_k},\left( \W_{\cdot \wedge \tau_k}(h_{j}) \right)_{j=1}^m\right) \E{
G_1(\zeta^{y,k})
\prod_{ l=k }^{ n } F_l\left(\W^k(e_l^y)\right)  }\Bigg|_{y=\Y_{\cdot\wedge\tau_k}}} .
\end{align*}
Then we apply Lemma~\ref{lem_independence_for_wk} and we denote by  $w_l$, $l=k,\dots,n$, a family of standard independent Brownian motions that do not depend on $\Y$ and $\W$.
\begin{align*}
E&=\E{ G_0\left(\Y_{\cdot \wedge \tau_k},\left( \W_{\cdot \wedge \tau_k}(h_{j}) \right)_{j=1}^m\right)
G_1(\zeta^k)} 
 \prod_{ l=k }^{ n } \E{F_l\left(w_l\right)},
\end{align*}
which achieves the proof of the first part of the statement. 

Furthermore, for every $l>k$,  we remark that $e_l$ and $\tau_l$ are $\G_k$-measurable because they are $\F^{\Y}_{\tau_l}$-measurable and $\F^{\Y}_{\tau_l}\subseteq\F^{\Y}_{\tau_k}\subseteq \G_k$. Then the process $\W^l_{t\wedge \tau_{k,l}}=\W_{(t\wedge \tau_{k,l})+\tau_l}-\W_{\tau_l}$, $t\geq 0$, is $\G_k$-measurable, and consequently, $\W^l_{\cdot \wedge \tau_{k,l}}(e_l)$ is also $\G_k$-measurable. This finishes the proof of the second part of the lemma.
\qed
\end{proof}

Next, we define the gluing map $\Gl:\mC_0[0,\infty)^2 \times [0,\infty) \to \mC_0[0,\infty)$ as follows
\begin{equation}
\label{def:map_Gl}
\Gl(x_1,x_2,r)(t)=x_1(t\wedge r)+x_2\left((t-r)^+\right), \quad t\geq 0,
\end{equation}
where $a^+:=a\vee 0$. It is easily seen that the map $\Gl$ is continuous and therefore measurable.

Since almost surely, $\W^k_t(e_l)=\W^l_{t+\tau_k-\tau_l}(e_l)-\W^l_{\tau_k-\tau_l}(e_l)$, $t\geq 0$, for every $l>k\geq 1$, 
 a simple computation shows that for every $l>k\geq 1$ almost surely 
\begin{equation} 
  \label{equ_gl_for_wk}
  \W^l(e_l)=\Gl\left(\W^l_{\cdot \wedge\tau_{k,l}}(e_l),\W^k(e_l),\tau_{k,l}\right),
\end{equation}
where $\tau_{k,l}:=\tau_k-\tau_l$.  

\begin{proof}[Proof of Lemma~\ref{lem_independence_of_y_and_wk}]
In order to prove this lemma, it is enough to show that for each $k \geq 1$, 
$\W^k(e_k)$ is independent of $\Y$, $\W^l(e_l)$, $l >k$. 

Let us denote by $\H_k$ be the complete $\sigma$-algebra generated by $\G_k$ and $\W^k(e_l)$, $l >k$. 
By Lemma~\ref{lem_properties_of_wk}, the process $\W^k(e_k)$ is independent of $\H_k$. 

Moreover for every $l>k$, using Lemma~\ref{lem_measurability_of_y}, $\Y$ and $\tau_{k,l}$ are $\G_k$-measurable, hence they are $\H_k$-measurable. By Lemma~\ref{lem_properties_of_wk} and  by the definition of $\H_k$, we also see that $\W^l_{\cdot \wedge\tau_{k,l}}(e_l)$ and $\W^k(e_l)$  are $\H_k$-measurable. By~\eqref{equ_gl_for_wk}, it follows that $\W^l(e_l)$ is $\H_k$-measurable for every $l>k$. 
Therefore $\Y$, $\W^l(e_l)$, $l >k$, are independent of $\W^k(e_k)$. 
\qed
\end{proof}

\subsection{Proof of Proposition~\ref{pro_uniqueness_of_coupling} }
\label{parag:proof_prop}

Let $\beta_k$, $k \geq 0$, be independent standard Brownian motions, \textit{independent} of $\X=(\Y,\W)$. 
Recall that $\pr_{\Y_t}^\bot e_k=\I_{\left\{ t\geq \tau_{k} \right\}}e_k$, $t\geq 0$, is a right-continuous $(\F^{\X}_t)$-adapted process in $L_2$. 
Thus we can define  for every $k\geq 0$ 
\begin{equation} 
  \label{equ_family_bk}
  B_k(t):=\beta_k(t \wedge \tau_k)+\int_{ 0 }^{ t } \I_{\left\{ s\geq \tau_{k}\right\}} e_k \cdot \mathrm d\W_s, \quad t\geq 0. 
\end{equation}
Since $\tau_0=+\infty$, we have in particular $B_0(t)= \beta_0(t)$, $t \geq 0$. 

\begin{lemma} 
  \label{lem_b_are_brownian_motions_independent_on_y}
  The processes $B_k$, $k\geq 0$, defined by~\eqref{equ_family_bk}, are independent standard Brownian motions.
\end{lemma}

\begin{proof} 
  The statement of this lemma directly follows from L\'evy's characterization of Brownian motion~\cite[Theorem~II.6.1]{Ikeda:1989}.
\qed
\end{proof}

We will now use the result of Lemma~\ref{lem_independence_of_y_and_wk} to prove the following lemma.

\begin{lemma} 
  \label{lem_independence_of_bk_and_gk}
  The processes $\Y$, $B_k$, $k\geq 0$, are independent.
\end{lemma}

\begin{proof} 
Since $B_0=\beta_0$ is independent of $\Y$ by definition and of $B_k$, $k\geq 1$, by Lemma~\ref{lem_b_are_brownian_motions_independent_on_y}, it is enough to prove that the processes $\Y$, $B_k$, $k \in [n]$, are independent, for any given $n$. 

Putting together~\eqref{def:Wkel}, \eqref{def:map_Gl} and~\eqref{equ_family_bk}, we have
  \[
    B_k=\Gl\left( \beta_k,\W^k(e_k),\tau_k \right), \quad k \in [n].
  \]
Since $\beta_k$, $k \in [n]$, is independent of $(\Y, \W)$ and using Lemma~\ref{lem_independence_of_y_and_wk}, we deduce that the processes $\Y$, $\beta_k$, $\W^k(e_k)$,  $k\in [n]$, are independent. Moreover, $\tau_k$, $k\in [n]$, are measurable 
with respect to $\F^{\Y}:=\sigma(\Y)$.   
Let $G_0:\mC([0,\infty),\Li) \to \R $, $F_k:\mC[0,\infty) \to \R $, $k\in [n]$, be bounded measurable functions.
We have
\begin{align*}
 \E{ G_0\left( \Y \right) \prod_{ k=1 }^{ n } F_k\left( B_k \right) }   &= \E{G_0\left( \Y \right) \E{ \prod_{ k=1 }^{ n } F_k\left( \Gl\left( \beta_k,\W^k(e_k),\tau_k \right) \right)\bigg| \F^{\Y} }}\\
    &= \E{G_0\left( \Y \right)\E{\prod_{ k=1 }^{ n } F_k\left( \Gl\left(\beta_k,\W^k(e_k),\tau_k^y \right) \right) }\bigg|_{y=\Y} }.
\end{align*} 
Note that if $w_1$ and $w_2$ are independent  standard Brownian motions and $r>0$, then the process $\Gl(w_1,w_2,r)$ is a standard Brownian motion. It follows that for any fixed $y \in \coal$, $\Gl\left( \beta_k,\W^k(e_k), \tau_k^y \right)$, $k\in [n]$, is a family of independent standard Brownian motions. Thus for every $y \in \coal$, 
\begin{align*}
\E{\prod_{ k=1 }^{ n } F_k\left( \Gl\left(\beta_k,\W^k(e_k),\tau_k^y \right) \right) }
=\prod_{ k=1 }^{ n }\E{ F_k\left( w_k \right) }, 
\end{align*}
where $w_k$, $k \in [n]$, denotes an arbitrary family of independent  standard Brownian motions. This easily implies the statement of the lemma, because $B_k$, $k \in [n]$, are independent standard Brownian motions by Lemma~\ref{lem_b_are_brownian_motions_independent_on_y}.
\qed
\end{proof}

Now, we  finish the proof of Proposition~\ref{pro_uniqueness_of_coupling}.

\begin{proof}[Proof of Proposition~\ref{pro_uniqueness_of_coupling}]
Define
\[
  B_t(h):=\sum_{ k=0 }^{ \infty } (h,e_k)_{L_2}B_k(t), \quad h \in L_2.
\]
Since $B_k$, $k\geq 0$, are independent Brownian motions that do not depend on $\Y$ and hence on $e_k$, $k\geq 1$, one can show similarly to the proof of Lemma~\ref{lem:shifted_W_is_BM} that the series converges in $\mC[0,\infty)$ almost surely for every $h \in L_2$, and $B_t$, $t\geq 0$, is a cylindrical Wiener process in $L_2$ starting at $0$. 

Moreover, $B$ is independent of $\Y$. 
Indeed, for any $n \geq 1$, for any  $h_1, \dots, h_n$ in $L_2$, for any bounded and measurable functions $F: \mC[0, \infty) ^n \to \R$ and $G: \mC([0,\infty),\Li) \to \R$, 
\begin{align*}
\E{F\left(B(h_1), \dots, B(h_n)\right) G\left(\Y\right)}
&= \E{\E{F\left(B(h_1), \dots, B(h_n)\right) \big| \F^{\Y}}G\left(\Y\right)} \\
&= \E{ \E{F\left(w_1, \dots, w_n\right)}   G\left(\Y\right) }\\
&= \E{F\left(B(h_1), \dots, B(h_n)\right)} \E{G\left(\Y\right)},
\end{align*}
where $w_k$, $k \in [n]$, denotes an arbitrary family of independent  standard Brownian motions.

Moreover, since $\pr_{\Y_t}^\bot e_k=\I_{\left\{ t\geq \tau_{k} \right\}}e_k$, we easily check that
\begin{align*}
\int_{ 0 }^{ t } \pr_{\Y_s}^\bot h \cdot \mathrm dB_s
=  \int_{ 0 }^{ t } \pr_{\Y_s}^\bot h \cdot \mathrm d\W_s &=  \int_0^t h \cdot \mathrm d\W_s - \int_{ 0 }^{ t } \pr_{\Y_s} h \cdot \mathrm d\W_s \\
&= \W_t (h) - (g , h)_{L_2}  - ( \Y_t - g,h )_{L_2}
\end{align*}
for all $t\geq 0$, which implies equality~\eqref{equ_uniqueness_coupling}.
\qed
\end{proof}

\appendix

\section{Appendix: Regular conditional probability}
\label{appendix:cond_prob}

\subsection{Definition}

Let $\bE$ be a Polish space and $\bF$ be a metric space. We consider random elements $X$ and $\xi$ in $\bE$ and $\bF$, respectively, defined on the same probability space $(\Omega, \F, \p)$. Let also $\B(\bE)$ (resp. $\B(\bF)$) denote the Borel $\sigma$-algebra on $\bE$ (resp. $\bF$) and $\cP(\bE)$ be the space of probability measures on $(\bE,\B(\bE))$ endowed with the topology of weak convergence.

\begin{definition}
\label{def:regular_conditional_probability}
A function $p: \B(\bE) \times \bF \to [0,1]$ is  a \textit{regular conditional probability of $X$ given $\xi$}  if
\begin{itemize}
\item[(R1)] for every $z \in \bF$, $p(\cdot,z)\in \cP(\bE)$;
\item[(R2)] for every $A \in \B(\bE)$, $z \mapsto p(A, z)$ is measurable;
\item[(R3)] for every $A \in \B(\bE)$ and $B \in \B(\bF)$, 
\begin{align*}
\P{X \in A,\ \xi \in B}=\int_B p(A,z)\; \mathbb P^\xi (\mathrm dz),
\end{align*}
where $\mathbb P^\xi:=\mathbb P \circ \xi^{-1}$ denotes the law of $\xi$.
\end{itemize} 
\end{definition}

Recall the following existence and uniqueness result (see e.g.~\cite[Theorem 6.3]{Kallenberg:2002}):
\begin{proposition}
\label{prop:kallenb}
There exists a regular conditional probability of $X$ given~$\xi$. Moreover, it is unique in the following sense: if $p$ and $p'$ are regular conditional probabilities of $X$ given $\xi$, then 
\begin{align*}
\p^\xi \left[ z\in \bF:\ p(\cdot,z)=p'(\cdot,z) \right]=1.
\end{align*}
\end{proposition}

\subsection{Proof of Lemma~\ref{lem_continuity_and_convergence} }
\label{parag:proof_of_lem_continuity_and_convergence}

We first recall  that the sufficiency of Lemma~\ref{lem_continuity_and_convergence} immediately follows from the continuous mapping theorem. 

We next prove the necessity. 
We first choose a family $\{ f_k,\ k\geq 1 \}\subset \mC_b(\bE)$ which strongly separate points in $\bE$. One can show that such a family exists since $\bE$ is separable (see also~\cite[Lemma~2]{Blount:2010}).  By~\cite[Theorem~4.5]{Ethier:1986} (or~\cite[Theorem~6]{Blount:2010} for weaker assumptions on the space $\bE$), any sequence $\{\mu_n\}_{n\geq 1}$ of probability measures on $\bE$ converges weakly to a probability measure $\mu$ if and only if
\[
  \int_{ \bE }f_k(x)\mu_n(\mathrm dx)\to\int_{ \bE }   f_k(x)\mu(\mathrm dx), \quad n \to \infty,     
\]
for all $k\geq 1$.

We define the following sets 
\begin{align*}
  A_m^{k,+}&= \left\{ z \in \bF: \int_{ \bE }   f_k(x)p(\mathrm dx,z)-\int_{ \bE } f_k(x)  \nu(\mathrm dx)\geq \frac{1}{ m }   \right\},\\
  A_m^{k,-}&= \left\{ z \in \bF: \int_{ \bE }  f_k(x) \nu(\mathrm dx)-\int_{ \bE }   f_k(x)p(\mathrm dx,z)\geq \frac{1}{ m }   \right\}
\end{align*}
for all $k\geq 1$ and $m\geq 1$. Let also $A_m^k=A_m^{k,+}\cup A_m^{k,-}$.

\begin{lemma} 
  \label{lem_technical_lemma_about_continuity_of_p}
  If for every $k\geq 1$ and $m\geq 1$ there exists $\delta_m^k>0$ such that 
  \begin{equation} 
  \label{equ_equality_for_amk_and_bmk}
  \p^{{\T(X)}}\left[A_m^k \cap B_m^k\right]=0,
  \end{equation}
  where $B_m^k$ is the ball in $\bF$ with center $z_0$ and radius $\delta_m^k$, then $p$ has a version continuous at $z_0$. Moreover, it can be taken as 
  \[
    p'(\cdot ,z)=
    \begin{cases}
      p(\cdot,z), & \mbox{ if } z \not\in \bigcup_{ k,m=1 }^{ \infty }\left( A_m^k \cap B_m^k \right), \\
      \nu, & \mbox{ otherwise.} 
    \end{cases}
  \]
\end{lemma}

\begin{proof} 
  We first remark that according to~\eqref{equ_equality_for_amk_and_bmk}, $p'=p$ $\p^{{\T(X)}}$-a.e. Next, let $z_n \to z_0$ in $\bF$ as $n \to \infty$. Without loss of generality, we may assume that $z_n \not\in \bigcup_{ k,m=1 }^{ \infty }\left( A_m^k \cap B_m^k \right)$ for all $n\geq 1$. Let $m\geq 1$ and $k\geq 1$ be fixed. Then there exists a number $N$ such that $z_n \in B_m^k$ for all $n\geq N$. Consequently, $z_n \not\in A_m^k$, $\forall n\geq N$, that yields
  \[
    \left| \int_{ \bE }   f_k(x)p(\mathrm dx,z_n)-\int_{ \bE }   f_k(x)\nu(\mathrm dx)   \right|< \frac{1}{ m }
  \]
  for all $n\geq N$. This finishes the proof of the lemma.
\qed
\end{proof}

We come back to the proof of Lemma~\ref{lem_continuity_and_convergence}. Let us assume that $p$ has no version continuous at $z_0$. Then, according to Lemma~\ref{lem_technical_lemma_about_continuity_of_p}, there exists $k\geq 1$ and $m\geq 1$ such that for every $\delta>0$ 
\[
  \p^{{\T(X)}}\left[ A_m^k \cap B_{\delta} \right]>0,
\]
where $B_{\delta}$ denotes the ball with center $z_0$ and radius $\delta$. 
Without loss of generality, we may assume that $\p^{{\T(X)}}\left[ A_m^{k,+} \cap B_{\delta} \right]>0$ for every $\delta>0$.
For every $n\geq 1$, let $\xi^n$ be a random element in $\bF$ with distribution
\[
  \p^{\xi^n}[A]=\int_{ \bF }q_n(z)\p^{{\T(X)}}[\mathrm dz]     , \quad A \in \B(\bF) ,
\]
where
\[
  q_n(z)=\frac{1}{ \p^{{\T(X)}}\left[ A_m^{k,+} \cap B_{ \frac{1}{ n }} \right]} \I_{A_m^{k,+} \cap B_{ \frac{1}{ n }}}(z), \quad z \in \bF.
\]
By the construction, $\p^{\xi^n}\ll \p^{{\T(X)}}$, $n\geq 1$. Moreover, it is easy to see that $\xi^n \to z_0$ in distribution as $n \to \infty$. But 
\[
  \E{ \int_{ \bE }   f_k(x)p(\mathrm dx,\xi^n)  } \not\to \int_{ \bE }   f_k(x)\nu(\mathrm dx),\quad n \to \infty. 
\]
Indeed, for every $n\geq 1$ the random element $\xi^n$ takes values almost surely in $A_m^{k,+}$, which implies that 
\[
 \E{ \int_{ \bE }   f_k(x)p(\mathrm dx,\xi^n)  }- \int_{ \bE }  f_k(x)\nu(\mathrm dx)\geq \frac{1}{ m },\quad n\geq 1.
\]
We have obtained the contradiction with assumption~\eqref{equ_value_of_p_along_xi_n}. This finishes the proof of Lemma~\ref{lem_continuity_and_convergence}.

\section{Some properties of MMAF}
\label{sec:some_properties_of_mmaf}

\subsection{Measurability of coalescing set}

We recall that the set $\DC$ denotes the space of \cdl functions from $(0,1)$ to $\mC[0,\infty)$ equipped with the Skorokhod distance, which makes it a Polish space. Set
\[
  \DI:=\left\{ y \in \DC:\ \forall 0<u<v<1,\ y_t(u)\leq y_t(v)\ \forall t\geq 0 \right\}.
\]
It is easily seen that $\DI$ a closed subspace of $\DC$. So, we will consider $\DI$ as a Polish subspace of $\DC$. Let 
\begin{align*}
  \DIL:&= \left\{ y \in \DI:\ \forall T \in \N,\ \exists K \in \N,\ \exists \delta \in  \Q_+, \max\limits_{ t \in \left[ \frac{1}{ T },T\right] }\|y_t\|_{L_{2+\delta}}\leq K\right\}\\
  &\quad \cap\left\{ y \in \DI:\ \|y_t-y_0\|_{L_2} \to 0,\ t \to 0 \right\}=:D^1 \cap D^2. 
\end{align*}

\begin{lemma} 
  \label{lem_measurability_between_d_and_l}
  For every $A\in\B(\DI)$ the set $A \cap \DIL$ is a Borel measurable subset of $\CL:=\mC([0,\infty),\Li)$.
\end{lemma}

\begin{proof} 
  First we are going to show that $\DIL$ is a subset of $\CL$. So, we take $y \in \DIL$ and check that $y$ is a continuous $L_2$-valued function. The continuity of $y$ at $0$ follows from the definition of $\DIL$. Let $t>0$ and $t_n \to t$ as $n\to\infty$. Without loss of generality, we may assume that $t_n \in [\frac{1}{ T },T]$ for some $T \in \N$ and all $n\geq 1$. We are going to show that $y_{t_n} \to y_t$ in $L_2$, $n \to \infty$. Let us note that the sequence $\{ y_{t_n} \}_{n\geq 1}$ is relatively compact, according to~\cite[Lemma~5.1]{Konarovskyi:EJP:2017} and the fact that $y_{t_n} \in \Li$, $n\geq 1$, are uniformly bounded in $L_{2+\delta}$-norm. This implies that there exists a subsequence $N \subseteq \N$ and $f \in \Li$ such that $y_{t_n} \to f$ in $L_2$ along $N$. On the other hand, $y_{t_n} \to y_t$ pointwise, that implies the equality $f=y_t$. Moreover, it yields that every convergent subsequence of $\{y_{t_n}\}_{n\geq 1}$ converges to $y_t$ in $L_2$. Using the relatively compactness of $\{ y_{t_n} \}_{n\geq 1}$, we can conclude that $y_{t_n} \to y_t$ in $L_2$ as $n\to\infty$. Thus, $y \in \CL$.

  Next, we will check that the set $\DIL$ is measurable in $\DI$. We fix $t\geq 0$ and make the following observation. For every $y \in \DI$ the real-valued function $y_t$ is non-decreasing on $(0,1)$. This implies that it has at most countable number of discontinuous points. Hence, by~\cite[Proposition~3.5.3]{Ethier:1986}, the convergence $y^n \to y$ in $\DI$ implies the convergence of $y^n_t \to y_t$ a.e. (with respect to the Lebesgue measure on $[0,1]$). Using Fatou's lemma, we get that the set 
  \begin{equation} 
  \label{equ_closability_of_norm_set}
  \Lambda(t,f,K,p):=\left\{ y \in \DI:\ \|y_t-f\|_{L_p}\leq K\right\} \quad \mbox{is closed in}\ \DI
  \end{equation}
  for every $K\geq 0$, $p\geq 2$ and $f \in L_p$. Hence the set 
  \[
    D^1=\bigcap_{ T=1 }^{ \infty } \bigcup_{ K=1 }^{ \infty } \bigcup_{ \delta \in \Q_+ } \bigcap_{t \in \left[ \frac{1}{ T },T \right]}   \Lambda(t,0,K,2+\delta)
  \]
  is Borel measurable in $\DI$. Using the standard argument and~\eqref{equ_closability_of_norm_set}, one can check the measurability of $D^2$. So, the set $\DIL=D^1 \cap D^2$ is Borel measurable in $\DI$.

  We claim that the identity map $\varPhi:\DIL \to \CL$ is Borel measurable. Indeed, let
  \[
    B^T_r(y):=\left\{ x \in \CL:\ \max\limits_{ t \in[0,T] }\|x_t-y_t\|_{L_2}\leq r \right\}.
  \]
  Then the preimage
  \[
    \varPhi^{-1}\left( B^T_r(y) \right)=\bigcap_{ t \in[0,T] }   \Lambda(t,y_t,r,2)
  \]
  is a closed set in $\DI$, by~\eqref{equ_closability_of_norm_set}.
  Since the Borel $\sigma$-algebra on $\CL$ is generated by the family $\left\{ B^T_r(y),\ T,r>0,\ y \in \CL \right\}$, $\varPhi$ is a Borel measurable function. Moreover, it is an injective map. So, using the Kuratowski theorem (see~\cite[Theorem~3.9]{Parthasarathy:1967}) and the fact that $A \cap \DIL \in \B(\DI)$, we obtain that the image $\varPhi(A \cap \DIL)=A \cap \DIL \in \B(\CL)$  for every $A \in \B(\DI)$.
\qed
\end{proof}

\begin{lemma}
\label{lemma:measurability_step}
Let $\coal$ be defined in Section~\ref{parag:mmaf_coal}. Then $\coal$ is a Borel measurable subset of $\CL$.
\end{lemma}

\begin{proof} 
  Let $\coal_D$ consists of all functions from $\DI$ which satisfies conditions (G2)-(G5) of the definition of $\coal$ in Section~\ref{parag:mmaf_coal}. Since every function $f \in \St$ has a finite $L_p$-norm for every $p\geq 2$, it is easily seen that 
  \[
    \coal_D \cap \DIL =\coal.
  \]
  Hence, according to Lemma~\ref{lem_measurability_between_d_and_l}, the statement of the lemma will immediately follow from the measurability of $\coal_D$ in $\DI$. However, this follows from the fact that the set $\coal_D$ can be determined via values of $y(u) \in \mC[0,\infty)$ for $u$ from a countable set $U$. We leave a detailed proof for the reader. 
\qed
\end{proof}

\subsection{Properties of MMAF}

Let $\{\Y(u,t),\ u \in (0,1),\ t \in [0,\infty)\}$ be a MMAF starting at $g \in \ltp$, and $\Y_t= \Y(\cdot, t)$, $t\geq 0$. 

\begin{lemma} 
  \label{lem_finiteness_of_l2_norm_of_mmaf}
  If $\|g\|_{L_{2+\eps}}<\infty$ for some $\eps>0$, then for every $T>0$ and $\delta \in \left( 0,\frac{ \eps }{ 2+\eps } \right)$ there exists $C_{T,\delta}$ such that
  \[
    \E{ \sup\limits_{ t \in [0,T] }\|\Y_t-g\|_{L_{2+\delta}}^{2+\delta} }\leq C_{T,\delta}\left(1+\|g\|_{L_{2+\eps}}\right).
  \]
\end{lemma}

\begin{proof} 
  In order to check the estimate, one needs to repeat the proof of \cite[Proposition~4.4]{Konarovskyi:EJP:2017} replacing the summation with the integration. 
\qed
\end{proof}

Recall that for every $f \in \St$, $N(f)$ denotes the number of steps of $f$. We write $N(f)= \infty$ for each non-decreasing \cdl function $f$ which does not belong to $\St$.
For any $y \in \mC([0,\infty),\Li)$, define 
  \[
    \tau^{y}_k= \inf \{ t\geq 0 : N(y_t) \leq k  \},\quad k\geq 0.
  \]

The following lemma states that a MMAF satisfies almost surely Property (G5) of Definition~\ref{definition_coal}.

\begin{lemma} 
  \label{lem_property_of_ny}
  Let $\Y_t$, $t \geq 0$, be a MMAF starting at $g$. Then
  \begin{equation} 
  \label{equ_different_times}
  \P{ \forall k< N(g), \tau_{k+1}^{\Y} < \tau_k^{\Y} }=1\quad\mbox{and}\quad \P{ \tau^{\Y}_1< +\infty }=1.
\end{equation}
\end{lemma}

\begin{proof} 
The proof of the statement follows from the fact that with probability one, three or more independent Brownian motions cannot meet at the same time, and from mathematical induction. 
\qed
\end{proof}

\begin{lemma} 
  \label{lem_tau_and_n}
  For every $y \in \coal$, $\beta>0$ and $n\geq 1$ one has 
  \[
    \sum_{ k=n }^{ \infty } \left( \tau_k^y \right)^{\beta}=\beta \int_{ 0 }^{ \tau_n^y } (N(y_t)-n) t^{\beta-1} \mathrm dt.
  \]
\end{lemma}

\begin{proof} 
  For simplicity of notation we will omit the superscript $y$ in $\tau_k^y$.  We write for $m>n$
  \begin{align*}
    \sum_{ k=n }^{ m } \tau_{k}^{\beta}&= \sum_{ k=n+1 }^{ m } (k-n)\left( \tau_{k-1}^{\beta}-\tau_k^{\beta} \right)+(m+1-n)\tau_m^{\beta}\\
    &= \sum_{ k=n+1 }^{ m } (N(y_{\tau_k})-n)\left( \tau_{k-1}^{\beta}-\tau_k^{\beta} \right)+(N(y_{\tau_{m+1}})-n)\tau_m^{\beta}\\
    &= \int_{ 0 }^{ \tau_n } \left( N(y_{t\vee\tau_{m+1}})-n \right)\mathrm dt^{\beta}= \beta\int_{ 0 }^{ \tau_n } \left( N(y_{t\vee\tau_{m+1}})-n \right)t^{\beta-1}\mathrm dt,
  \end{align*}
Hence, the statement of the lemma follows from the monotone convergence theorem.
\qed
\end{proof}

\begin{lemma}
\label{lem:stopping_time_mmaf}
Let $\Y$ be a MMAF starting at $g \in \ltp$. Then for every $\beta> \frac{1}{ 2 }$,  $\sum_{k=1}^\infty (\tau_k^\Y)^\beta< +\infty$ almost surely. 
\end{lemma}

\begin{proof}
  Let $g \in L_{2+\eps}$ for some $\eps>0$. In order to prove the lemma, we will use the estimate 
  \[
    \E{ N(\Y_t)}\leq \frac{ C_{\eps,T} }{ \sqrt{ t } }\left( 1+\|g\|_{L_{2+\eps}} \right), \quad t \in (0,T],
  \]
  from~\cite[Remark~4.6]{Konarovskyi:EJP:2017}, where $C_{\eps,T}$ is a constant depending on $\eps$ and $T>0$. Take an arbitrary number $T>0$ and estimate for $\beta> \frac{1}{ 2 }$ 
\begin{align*}
  \E{ \int_{ 0 }^{ \tau_1^{\Y}\wedge T } N(\Y_t)t^{\beta-1} \mathrm dt}&\leq  \int_{ 0 }^{  T } \E{N(\Y_t)}t^{\beta-1} \mathrm dt\\
  &\leq C_{\eps,T}\left( 1+\|g\|_{L_{2+\eps}} \right) \int_{ 0 }^{ T } t^{\beta- \frac{3}{ 2 }}\mathrm dt < +\infty.
\end{align*}
Thus $\int_{ 0 }^{ \tau_1^{\Y} \wedge T } N(\Y_t)t^{\beta-1} \mathrm dt<\infty$ almost surely, for all $T>0$. 
Since  $\tau_1^{\Y}< \infty$ almost surely by~\eqref{equ_different_times},  $\int_{ 0 }^{ \tau_1^{\Y} } N(\Y_t)t^{\beta-1} \mathrm dt<\infty$ almost surely.
Thus, the statement of the lemma follows directly from Lemma~\ref{lem_tau_and_n}.
\qed
\end{proof}

\subsection{Uniqueness of solutions to deterministic equation}
\label{sub:uniqueness_of_solutions_to_deterministic_equation}

Denote $\bE^g:=\left\{x \in \mC\left([0,\infty),L_2(g)\right):\ x_0=g\right\}$ and
\[
\coalex:= \left\{ x \in \bE^{g}:\
    \begin{array}{l}
      \forall u,v \in (0,1)\ \forall s\geq 0,\ x_s(u)=x_s(v)\\
	   \mbox{implies}\ x_t(u)=x_t(v),\ \forall t\geq s 
    \end{array}\right\}.
\]
Let $g \in \St$ be fixed and $N(g)=n$. For functions $x \in \mC\left([0,\infty),L_2(g)\right)$ and $y \in \coalex$, we introduce the integral
\begin{equation} 
  \label{equ_definition_of_the_integral}
  \int_{ 0 }^{ t } \pr_{y_s} \mathrm dx_s= \sum_{ k=1 }^{ n } \bigg( \pr_{y_{\tau_{k}^y\wedge t}}x_{\tau_{k-1}^y\wedge t} -\pr_{y_{\tau_{k}^y\wedge t}}x_{\tau_{k}^y\wedge t} \bigg). 
\end{equation}

\begin{lemma} 
  \label{lem_solution_to_deterministic_equation}
  For every $x \in \mC\left([0,\infty),L_2(g)\right)$, there exists a unique $y \in \coalex$ such that 
  \begin{equation} 
  \label{equ_deterministic_equation}
    y_t=g+\int_{ 0 }^{ t } \pr_{y_s}\mathrm dx_s, \quad t\geq 0. 
  \end{equation}
\end{lemma}
 
\begin{proof} 
  Without loss of generality, we assume that $x_0=g$. The function $y \in \coalex$ can be constructed step by step. First take $\sigma_0=0$ and $\tilde{y}^0_t=g$, $t\geq 0$. Then set
  \[
    \tilde{y}^{k}_t:=\tilde{y}^{k-1}_{\sigma_{k-1}}+\pr_{\tilde{y}^{k-1}_{\sigma_{k-1}}}x_t-\pr_{\tilde{y}^{k-1}_{\sigma_{k-1}}}x_{\sigma_{k-1}}, \quad t\geq \sigma_{k-1},
  \]
  and 
  \[
    \sigma_k:=\inf\left\{ t>\sigma_{k-1}:\ \dim L_2(\tilde{y}^k_t)< \dim L_2(\tilde{y}^{k-1}_{\sigma_{k-1}}) \right\}
  \]
  for all $k \in [n-1]$.  Remark that $\dim L_2(\tilde{y}^k_{\sigma_k}) \in [n-k]$ for each $0\leq k \leq n-1$. We set $\sigma_n=+\infty$. 
  The function $y$ can be defined as 
    \[
      y_t=\tilde{y}^{k}_t \quad \mbox{for}\  t \in [\sigma_{k-1},\sigma_k),\quad k \in [n].
    \]
    By construction, $y$ belongs to $\coalex$, satisfies~\eqref{equ_deterministic_equation} and is uniquely determined.
\qed
\end{proof}

\begin{corollary} 
  \label{cor_existence_of_solution_to_the_equation}
  Let $W$ be a cylindrical Wiener process in $L_2(g)$ starting at $g$. Then there exists a unique $(\F^W_t)$-adapted process $Y_t$, $t\geq 0$, such that 
  \[
    Y_t=g+\int_{ 0 }^{ t } \pr_{Y_s}\mathrm dW_s, \quad Y_t \in L_2^{\uparrow}, \quad t\geq 0, 
  \]
  where $(\F^W_t)_{t\geq 0}$ is the filtration generated by $W$.
\end{corollary}

\begin{proof} 
  The statement of the lemma directly follows from Lemma~\ref{lem_solution_to_deterministic_equation} and the fact that $L_2^{\uparrow}$-valued continuous martingales starting from $g$ belongs to $\coal$ almost surely (see~\cite[Proposition~2.2]{Konarovskyi:EJP:2017}). 
\qed
\end{proof}

\subsection{On map $\varphi$}

In Remark~\ref{explanation_phi}, we announced the following result. 
\begin{lemma}
\label{lem:set_coal_as_preimage}
For every $y \in \coal$ and $z=(z_k)_{k \geq 1} \in \mC_0[0, \infty)^\N$ define similarly to~\eqref{formula_phi} 
\[
  \varphi_t(y,z)= y_t+\sum_{k=1}^\infty e_k^y \I_{\left\{ t \geq \tau_k^y \right\}} z_k(t - \tau_k^y), \quad t \geq 0,
\]
if the series converges in $\mC([0,\infty),L_2)$. 
Then for each $y \in \coal$, $\varphi(y,z)$ belongs to $\coal$ if and only if $z=0$. 
\end{lemma}

\begin{proof}
It is obvious that $\varphi(y,0)=y \in \coal$.

We assume now that $\varphi(y,z)$ belongs to $\coal$ and prove that $z=0$. Set
\[
  \gamma(y,z)= \sum_{k=1}^\infty e_k^y \I_{\left\{ t \geq \tau_k^y \right\}} z_k(t - \tau_k^y), \quad t \geq 0,
\]
and show that $\gamma(y,z)=0$. This will immediately imply $z=0$.

{\it Step I.}
Let $k \geq 1$ be fixed.  By~\eqref{eky}, there exist $a<b<c$ such that 
\begin{align*}
e_k^y = \frac{1}{\sqrt{c-a}}  \left(\sqrt{\frac{c-b}{b-a}} \I_{[a,b)} -\sqrt{\frac{b-a}{c-b}} \I_{[b,c)} \right).
\end{align*}
The goal of this step is to show that $\varphi_{\tau_k^y}(y,z)(u) = y_{\tau_k^y}(u)$ for every $u \in [a,c)$, in other words, that $\gamma_{\tau_k^y}(y,z)$ is equal to zero on the interval $[a,c)$. 

By the construction of $\tau_k^y$ and $e_k^y$, $y_{\tau_k^y}$ is constant on the interval $[a,c)$. Furthermore, since $\varphi(y,z) \in \coal$, $\varphi_{\tau_k^y}(y,z)$ belongs to $\Li$. Hence, we can deduce that $\gamma_{\tau_k^y}(y,z) =  \varphi_{\tau_k^y}(y,z) -y_{\tau_k^y}(y,z)$ is non-decreasing on $[a,c)$, as a difference of a non-decreasing function and a constant function. Furthermore, 
\begin{align*}
  \gamma_{\tau_k^y}(y,z)= \sum_{l=k}^\infty e_l^y  z_l(\tau_k^y - \tau_l^y) = \sum_{l=k+1}^\infty e_l^y  z_l(\tau_k^y - \tau_l^y),
\end{align*}
since $z_k(0)=0$. Hence, $\gamma_{\tau_k^y}(y,z)$ belongs to $\spann \left\{e_l^y,\ l \geq k+1 \right\}$, whereas $\I_{[a,b)}$ and $\I_{[a,c)}$ both belong to $\spann \left\{e_l^y,\ l \leq k\right\}$. Indeed, $\I_{[a,c)} \in L_2(y_{\tau_k^y})= \spann \left\{ e_l^y,\ l < k\right\}$, by Lemma~\ref{lem:onb}, and $\I_{[a,b)}  \in \spann \left\{\I_{[a,c)}, e_k^y\right\}$.  Recall that $\{e_l^y,\ l \geq 0\}$ is an orthonormal basis of $L_2$. Thus,
\begin{align*}
  \left(\gamma_{\tau_k^y}(y,z), \I_{[a,b)}\right)_{L_2} = \left(\gamma_{\tau_k^y}(y,z), \I_{[a,c)}\right)_{L_2}=0.
\end{align*}
So, we can deduce that $u \mapsto (\gamma_{\tau_k^y}(y,z), \I_{[a,u)})_{L_2}$ is a convex function on $[a,c]$ which vanishes at $a$, $b$ and $c$. Thus, it is zero everywhere on $[a,c]$. In particular, $\gamma_{\tau_k^y}(y,z)(u)=0$ for every $u \in (a,c)$. Consequently, $\varphi_{\tau_k^y}(y,z)(u) = y_{\tau_k^y}(u)$ for every $u \in (a,c)$. The equality also holds for $u=a$, by the right-continuity of $\varphi_{\tau_k^y}(y,z)$ and $y_{\tau_k^y}$. 

{\it Step II.} Now let $t >0$ be fixed. By Property (G3) of the definition of $\coal$ in Section~\ref{parag:mmaf_coal}, $y_t$ belongs to $\St$, and thus, 
\[
  y_t(u) = \sum_{j=1}^n y_j \I_{[a_j,c_j)}(u),
\]
for pairwise distinct $y_j$, $j\in [n]$. Fix $j \in [n]$. By coalescence Property (G4), there exists $k \geq 1$ such that $u \mapsto y_s(u)$ is constant on $[a_j,c_j)$ for every $s \geq \tau_k^y$ and non-constant on $[a_j,c_j)$ for every $s < \tau_k^y$. By Step I, $y_{\tau_k^y}= \varphi_{\tau_k^y}(y,z)$ on $[a_j,c_j)$. Thus, $\varphi_{\tau_k^y}(y,z)$ is constant on $[a_j,c_j)$. By Property (G4) again, now applied to $\varphi(y,z)$, $\varphi_t(y,z)$ is constant on $[a_j,c_j)$ due to $t \geq \tau_k^y$.  As the difference of two constant functions, $\gamma_t(y,z)$ is also constant on $[a_j,c_j)$. Moreover, by the construction of $\gamma$ and Lemma~\ref{lem:onb}, $\gamma_t(y,z)$ is orthogonal to $L_2(y_t)$. Hence $\gamma_t(y,z)$ is also orthogonal to $\I_{[a_j,c_j)}$. Therefore, we can conclude that $\gamma_t(y,z)= 0 $ on $[a_j,c_j)$. Since $j \in [n]$ and $t>0$ were arbitrary, we deduce that $\gamma_t(y,z)=0$ on $[0,1)$ for every $t >0$. This finishes the proof of the lemma.
\qed
\end{proof}

\begin{acknowledgements}
The authors are very grateful to Max von Renesse for very useful discussions and suggestions. The first author also thanks Andrey A. Dorogovtsev for his helpful comments and the interest  to the studying problem. 
The first author was partly supported by the Deutsche Forschungsgemeinschaft (DFG, German Research Foundation) – SFB 1283/2 2021 – 317210226.  
\end{acknowledgements}


\end{document}